\documentclass[10 pt]{amsart}

\usepackage{latexsym, amsmath, amssymb, longtable, booktabs, amscd, microtype, booktabs, mathrsfs, array}
\usepackage[square]{natbib}
\usepackage{graphicx}
\usepackage[dvipsnames]{xcolor}
\usepackage{caption}
\usepackage[margin=2.5cm]{geometry}

\usepackage{hyperref}
\hypersetup{
    colorlinks,
    citecolor=blue,
    filecolor=black,
    linkcolor=red,
    urlcolor=black
}	
\bibpunct{[}{]}{,}{n}{}{;} 

\numberwithin{equation}{section}

\newcommand{\prid}{\mathfrak{p}}
\newcommand{\zz}{\mathbb{Z}}
\newcommand{\cc}{\mathbb{C}}
\newcommand{\qq}{\mathbb{Q}}
\newcommand{\ff}{\mathbb{F}}

\newcommand{\rr}{\mathbb{R}}
\newcommand{\cX}{\mathscr{X}}
\newcommand{\cO}{\mathcal{O}}
\newcommand{\fin}{\text{fin}}
\newcommand{\tcX}{\widetilde{\cX}}
\newcommand{\hfal}[1]{h_{\mathrm{Fal}}\left( #1 \right)}
\newcommand{\hNT}{h_{\mathrm{NT}}}
\newcommand{\gcan}{\mathfrak{g}_{\mathrm{can}}}
\newcommand{\cY}{\mathscr{Y}}
\newcommand{\cZ}{\mathscr{Z}}
\newcommand{\tcZ}{\widetilde{\cZ}}
\newcommand{\cW}{\mathscr{W}}

\newcommand{\rhat}{\widehat{\cO}_{K}}
\newcommand{\fancyp}{\mathfrak{\pi}}
\newcommand{\dfal}{\delta_{\mathrm{Fal}}}

\newcommand{\mucan}{\mu_{\mathrm{can}}}
\newcommand{\selfinter}{\ov{\omega}^2_{X_0(p^2)}}
\newcommand{\genus}{g_{X_0(p^2)}}
\newcommand{\jac}{\mathrm{Jac}(X_0(p^2))}
\newcommand{\ar}[2]{\left\langle #1, #2 \right\rangle_{\mathrm{Ar}}}

\newcommand{\liu}{MR1917232}
\newcommand{\fal}{MR740897}
\newcommand{\edi}{MR1056773}
\newcommand{\mor}{MR1022303}
\newcommand{\jorg}{MR2521110}
\newcommand{\zhang}{MR1207481}
\newcommand{\ara}{MR0466150}
\newcommand{\lang}{MR969124}
\newcommand{\au}{MR1437298}
\newcommand{\may}{MR3232770}

\DeclareMathOperator{\pic}{Pic}
\DeclareMathOperator{\Div}{div}
\DeclareMathOperator{\spec}{Spec}

\newcommand\N{\mathbb{N}}

\newcommand{\Ga}{\Gamma}

\newcommand{\de}{\delta}

\newcommand{\ov}{\overline}

\newtheorem{thm}{Theorem}
\numberwithin{thm}{section}
\newtheorem{theorem}[thm]{Theorem}
\newtheorem{corollary}[thm]{Corollary}

\newtheorem{proposition}[thm]{Proposition}
\newtheorem{lemma}[thm]{Lemma}
\newtheorem{definition}[thm]{Definition}
\theoremstyle{remark}

\theoremstyle{definition}

\theoremstyle{remark}

\theoremstyle{remark}

\makeatletter
\def\imod#1{\allowbreak\mkern10mu({\operator@font mod}\,\,#1)}
\makeatother

\title{Semi-stable models of modular Curves $X_0(p^2)$ and some arithmetic applications}

\author{Debargha Banerjee}
\author{Chitrabhanu Chaudhuri}
\thanks{The first named author was partially 
supported by the SERB grants YSS/2015/001491 and MTR/2017/000357. The second named author was partially supported by 
the DST-INSPIRE grant IFA-13 MA-21}

\address{INDIAN INSTITUTE OF SCIENCE EDUCATION AND RESEARCH, PUNE, INDIA}

\begin{document}
\begin{abstract}
  In this paper, we compute the semi-stable models of modular curves $X_0(p^2)$ 
  for odd primes $p > 3$ and  compute the Arakelov self-intersection numbers of the 
  relative dualising sheaves for these models. We give two arithmetic applications of our 
  computations. In particular, we give an effective version of the Bogomolov  
  conjecture following the strategy outlined by Zhang and find the stable 
  Faltings heights of the arithmetic surfaces corresponding to these modular curves. 
\end{abstract}


\maketitle
\setcounter{tocdepth}{1}
\tableofcontents{}

\section{Introduction}
\label{Intro}
Let $K$ be a number field and $X$ a smooth projective curve over $K$ with genus $g_X >0$. 
Let $\cO_K$ be the ring of integers of $K$ and $\cX/\cO_K$ the minimal regular model for $X/K$. 
The model $\cX/\cO_K$ is a regular scheme of dimension 2 with a map $\cX \to \spec \cO_K$ hence an arithmetic
surface in the sense of Liu \cite{\liu}. There is an intersection theory for invertible 
sheaves on arithmetic surfaces due to Arakelov (see Section~\ref{sec:self-inter} for a brief summary). 
We denote this intersection pairing by $\ar{\ \cdot \ }{\ \cdot \ }$. Let $\omega_{\cX/\cO_K}$ be the 
relative dualising sheaf of $\cX/\cO_K$ over $\spec \cO_K$. The quantity
\begin{equation*}
  \ov{\omega}^2_{\cX} = 
  \frac{\ar{\ov\omega_{\cX/\cO_K}}{\ov\omega_{\cX/\cO_K}}}{[K:\qq]}
\end{equation*}
is an invariant for $X$ independent of the field $K$ if $\cX$ is semi-stable. This means
for any field extension $K'/K$ if we take $X' = X \times_{\spec K} \spec K'$ and $\cX'/\cO_{K'}$ is the
minimal regular model of $X'/K'$ then $\overline{\omega}^2_{\cX'} = \overline{\omega}^2_{\cX}$. 
We call this quantity the stable arithmetic self-intersection number for $X$ and denote it by $\ov{\omega}^2_{X}$. 

For the rest of the paper, we fix $K$ to be the number field
\begin{equation} \label{eq:field}
  K = \qq \Big(\sqrt[r]{p}, \zeta_{p+1} \Big),
\end{equation}
where $r = (p^2 -1)/2$ and $\zeta_{p+1}$ is a primitive $(p+1)$ - th root of unity. 

The main objective of this paper is to calculate the stable arithmetic self-intersection 
$\selfinter$ for the modular curves $X_0(p^2)/K$. The modular curve $X_0(p^2)$ is a compactified moduli 
space of elliptic curves along with some extra structure (see \cite{MR2112196}).

In Section \ref{sec:semistable} following Edixhoven \cite{\edi}, we show that  the minimal regular model 
$\cX_0(p^2)/\cO_K$ is semi-stable. Stable models for these curves over $\qq_p$ were already 
computed by Edixhoven. However, the stable models are not regular. We resolve singularities and blow down 
all possible rational components in the special fibers, without introducing singularities, and thus obtain the 
minimal regular models. The cases of $p \in \{5, 7, 13\}$ are dealt with separately in the appendix.
We expect that the semi-stable models will be useful to compute the index of the Eisenstein ideals, 
and prove the Ogg's conjecture for the modular curves of the form  $X_0(p^2)$ following the 
strategy of Mazur~\cite{MR488287}.

Let $\genus$ be the genus of $X_0(p^2)$. Using the geometry of the semi-stable models and the results of \cite{ddc1}, 
we obtain the following asymptotic formula (Theorem \ref{thm:self-inter}) in \S
~\ref{sec:self-inter}: 
\begin{equation*}
  \selfinter = 2 \genus \log p^2 + o(p^2 \log (p^2)).
\end{equation*}
The basic strategy to prove the above mentioned theorem is same as that of similar theorems 
proved for  modular curves of the form $X_0(N)$ (\cite{MR1437298, MR1614563}), $X_1(N)$ (\cite{MR3232770})
or $X(N)$ (\cite{Grados:Thesis}) when $N$ is square-free.
For Fermat's curves, similar theorems were proved in ~\cite{Curilla:Thesis}. Note that effective bounds 
on self-intersection numbers  are given for general arithmetic surfaces in \cite{Kuhnupperbound}, \cite{MR3581222}.
As a corollary of Theorem \ref{thm:self-inter} we obtain an asymptotic expression for the stable Faltings 
height $h_{\mathrm{Fal}}$ of the Jacobian $\jac/\qq$ of $X_0(p^2)/\qq$. This is also the arithmetic degree of the direct 
image of $\overline{\omega}_{\cX_0(p^2)/\cO_K}$ onto $\spec \cO_K$. In Corollary \ref{cor:falt} we show that
\begin{equation*}
  \hfal{\jac} = \frac{1}{6} \genus\log(p^2) + o(\genus\log(p^2)).
\end{equation*}

Next in Section \ref{sec:bogomolov} we prove an effective Bogomolov conjecture for the particular case of modular 
curves of the form $X_0(p^2)$ with $p$ --- an odd prime --- by using Zhang's proof of the general effective Bogomolov
conjecture~\cite[Theorem 5.6]{MR1207481}. In the above mentioned fundamental paper, the bounds are achieved in 
terms of the admissible self-intersection number. This number, however, depends on the semi-stable model 
of the particular modular curve $X_0(p^2)$. The computation of this number involves the geometry of the 
special fibers of $\cX_0(p^2)/\cO_K$ which is quite complicated. We wish to draw the attention of the reader 
to the seminal work of L. Szpiro \cite[Theorem 3]{MR1106917}, \cite{MR801928} for an introduction 
to this subject. 

Let $D= \infty$ be the divisor of degree $1$ corresponding to the cusp $\infty \in X_0(p^2)(\qq)$.
This divisor gives an embedding of the curve into its Jacobian $\varphi_D: X_0(p^2) \to \jac$. 
Let $\hNT$ be the N\'{e}ron-Tate height on $\jac$.

\begin{theorem} \label{thm:effectivebogomolov}
For a sufficiently large prime $p$, 
the set
\[
  \left\{x \in X_0(p^2)(\ov\qq) \mid \hNT (\varphi_D(x)) <\left(\frac{1}{2}-\epsilon\right) \log(p^2)  
  \right\}
\]
is finite  but the set 
\[
  \left\{x \in X_0(p^2)(\ov\qq) \mid \hNT (\varphi_D(x)) \leq \left(1+\epsilon\right) \log(p^2) 
  \right\}
\]
is infinite. 
\end{theorem}
The general Bogomolov conjecture has been proved by Ullmo \cite{MR1609514} using ergodic theory; however 
the proof is not effective.

In \cite{MR3207580}, various bounds are provided for the  Arakelov self-intersection numbers, admissible 
self-intersection numbers, and stable Faltings heights for any smooth projective curve. 
In this article, we provide a precise asymptotic expression in the particular case of modular curves of 
the form $X_0(p^2)$. In a recent article~\cite{Parent}, Pierre Parent has given a bound on the height of  points of
$X_0(p^2)(K)$ (for a quadratic field $K$) assuming Brumer's conjecture. It will be intriguing to relate our  
results with that of Parent's.  In a future direction, we hope to generalize our work to $X_0(p^3)$ using 
the semi-stable models constructed by McMurdy-Coleman~\cite{MR2661537}.  We hope to use the work of 
de-Shalit~\cite{MR1279607} to achieve this. 

\subsection{Acknowledgments}
  The authors wish to express sincere gratitude to the the website mathoverflow.net and specially 
  the user j.c. for answering questions regarding circuit reduction of a metrized graph. The authors 
  are greatly indebted to  Professors Jurg Kramer, Robin De Jong, Bas Edixhoven, Ulf K\"{u}hn, and 
  Shou-Wu Zhang for their continuous encouragement and email correspondence. We are grateful to the 
  anonymous referee for a thorough and careful reading of our paper that helped us to improve the 
  mathematical content of the paper.
\section{Semistable models} \label{sec:semistable}
Let  $K$ be the number field  as defined in the introduction. We consider the intermediate field 
$K' = \qq (\sqrt[r]{p})$ with $r = (p^2 -1)/2$. The extension $K'/\qq$ is totally ramified and there
is a prime ideal $\prid$ of $\cO_{K'}$ such that $\prid^{(p^2-1)/2} = p \cO_{K'}$. 
In $\cO_K$, the ideal $\prid $ splits into a product of $s = \varphi(p+1)/2$
distinct primes ideals
\begin{equation} \label{eq:primes}
  \prid \c = \prid_1 \cdots \prid_s.
\end{equation}
Moreover, since $K/K'$ is Galois the primes $\prid_i$'s are all Galois conjugates. Hence it follows that
$\cO_K / \prid_i \cong \ff_{p^2}$. 

Although it is not made explicit, the following theorem is a direct consequence of the results of Section 2 
of Edixhoven \cite{\edi}. 

\begin{theorem}[Edixhoven]
  The minimal regular model $\cX_0(p^2)/\cO_K$ of $X_0(p^2)$ is semi-stable.
\end{theorem}

This section is devoted to the study of this minimal regular model $\cX_0(p^2)/\cO_K$.
For any non-zero prime ideal $\mathfrak{q}$ of $\cO_K$ with $\mathfrak{q} \neq \prid_i$ the fiber 
$\cX_0(p^2)_{\mathfrak{q}}$ is a smooth curve of genus $\genus$ over the residue field $\cO_K/ \mathfrak{q}$. 
However, the fibers $\cX_0(p^2)_{\prid_i}$ are not smooth. Since $\prid_i$'s are Galois conjugates, 
all these special fibers are isomorphic and we shall determine these special fibers.

We start with the regular model $\tcX_0(p^2)/\zz$ of the modular curve, constructed by Edixhoven \cite{\edi}. 
To this we apply the procedure of section 2 in Edixhoven to obtain a stable model, and finally we resolve  the 
singularities. 

Let us spell out the procedure explicitly. There are $k = \lfloor  \frac{p}{12} \rfloor$ points of triple 
intersection in the special fiber 
\begin{equation*}
  \tcX_0(p^2) \times_{\spec \zz} \spec \ff_p .
\end{equation*}
We blow up at these points to get an arithmetic surface $\cY/\zz$ with normal crossings. 
Each triple intersection point gives us a projective line of multiplicity 
$p+1$ after blow up; we name these new components $L_1, \ldots, L_k$. Below is a picture of the special fiber 
of $\cY$ for different primes $p$. Each irreducible component is a projective line.
\begin{figure}[!htb]
  \centering
  \includegraphics[scale=0.5]{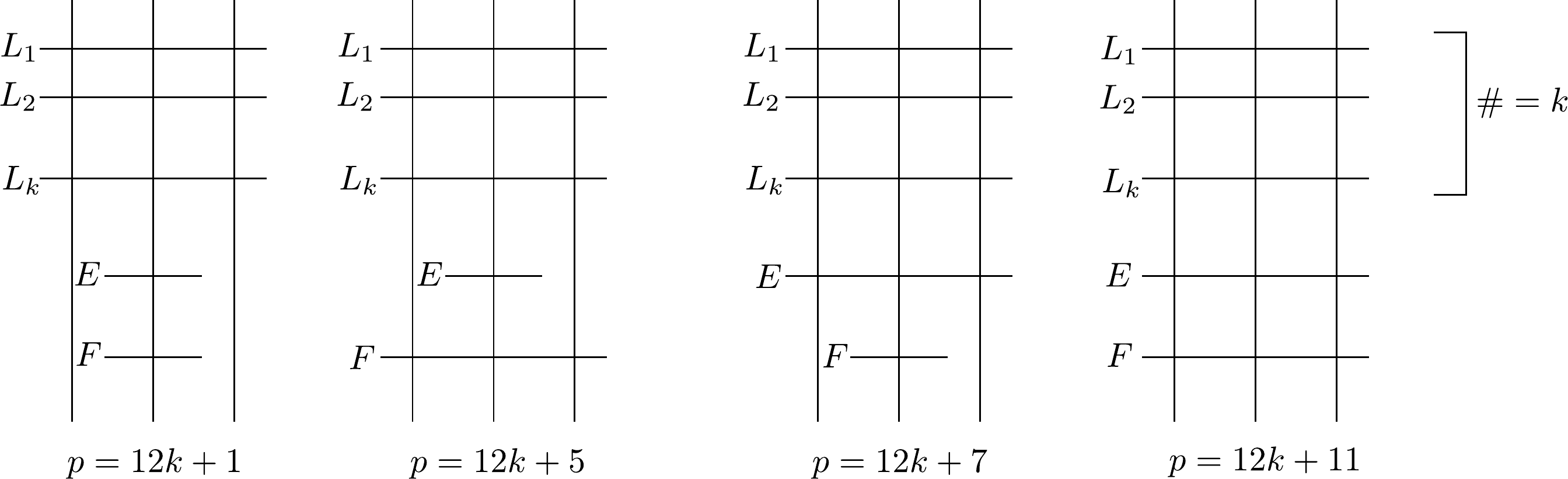}
\end{figure}
Next we base change to $\cO_K$, that is we let 
\begin{equation*}
  \cZ = \cY \times_{\spec \zz} \spec \cO_K.
\end{equation*}
Note that $\cZ_{\prid_i}$ for $i = 1, \ldots, \varphi(p+1)/2$ are the special fibers of $\cZ/\cO_K$, and we have
\begin{equation*}
  \cZ_{\prid_i} 
  \cong \cY_{(p)} \times \spec \ff_{p^2}.
\end{equation*}
Let $\tcZ/\cO_K$ be the normalisation of $\cZ/\cO_K$. Note that $\tcZ$ is not regular, so we calculate the de-singularization 
and call it $\cW/\cO_K$. Finally we blow down the rational components with self-intersection $-1$ in the special fibers 
of $\cW/\cO_K$ to get the minimal regular model $\cX_0(p^2)/\cO_K$. Since all the special fibers are isomorphic we fix 
one prime ideal $\prid_i$ as in \eqref{eq:primes} and look at the fiber over that.

We first do the calculation locally, that is we determine the complete local rings of $\tcZ$. Let 
$z \in \cZ_{\prid_i}$ be a closed point of the special fiber and $\widehat{\cO}_{\cZ,z}$ the complete local 
ring of that point. Then we use the following result from \cite[Chapter IV, 7.8.2 and 7.8.3]{EGA} or 
\cite[Lemma 0CBM]{stacks-project}
\begin{equation*}
  \tcZ \times_{\cZ} \spec \widehat{\cO}_{\cZ, z} = \text{Normalization of } \spec (\widehat{\cO}_{\cZ,z}). 
\end{equation*}

Calculations are dependent on the class of $p$ in $\zz/12\zz$. 

\subsection{Case 1: $p=12k + 1$} \label{subsec:1mod12}
Let $\rhat$ be the completion of $\cO_K$ at $\prid_i$. We can find a uniformizer $\fancyp$
of $\rhat$ such that $\fancyp^{(p^2-1)/2} = p$. Let $z$ be a closed point of the special fiber $\cZ_{\prid_i}$ then
\begin{equation*}
  \widehat{\cO}_{\cZ,z} \cong \frac{\rhat[[x,y]]}{(x^a y^b - \fancyp^{(p^2-1)/2})},
\end{equation*}
where at the double points of the special fiber
\begin{equation*}
  (a,b) \in \{(p+1,1)\} \cup \left(\left\{p+1, \frac{p-1}{2},\frac{p-1}{3} \right\} \times \{p-1\}\right)
\end{equation*}
and at smooth points of the special fiber
\begin{equation*}
  (a,b) \in \left\{1,p+1,p-1,\frac{p-1}{2}, \frac{p-1}{3} \right\} \times \{ 0 \}.
\end{equation*}
We want to calculate the complete local rings of the pre-image of these points under the normalisation 
morphism. All the cases have been discussed in \cite[Section 2.2]{MR1056773} except when 
$(a,b) \in \{\frac{p-1}{2},\frac{p-1}{3}\} \times \{0,p-1\}$. We quote the relevant results from Edixhoven.

\begin{itemize} \itemsep10pt
\item If $(a,b) = (1,0)$ the complete local ring is regular hence normal.

\item If $(a,b) = (p-1,0)$
      \begin{equation*}
        \tcZ \times_{\cZ} \spec \widehat{\cO}_{\cZ,z}
        = \bigsqcup_{\zeta \in \mu_{p-1}(\rhat)} \spec \rhat[[x,y]] /(x - \zeta \fancyp^{(p+1)/2}).
      \end{equation*}

\item If $(a,b) = (p+1,0)$
      \begin{equation*}
        \tcZ \times_{\cZ} \spec \widehat{\cO}_{\cZ,z}
        = \bigsqcup_{\zeta \in \mu_{p+1}(\rhat)} \spec \rhat[[x,y]] /(x - \zeta \fancyp^{(p-1)/2}).
      \end{equation*}

\item If $(a,b) = (p+1,1)$,
      \begin{equation*}
        \tcZ \times_{\cZ} \spec \widehat{\cO}_{\cZ,z}
        = \spec \rhat[[x,y]] /(xu - \fancyp^{(p-1)/2})
      \end{equation*}
      with $y=u^{p+1}$. The normalisation is not regular.

\item If $(a,b) = (p+1,p-1)$,
      \begin{equation*}
        \tcZ \times_{\cZ} \spec \widehat{\cO}_{\cZ,z}
        = \spec \rhat[[u,v]]/(uv-\fancyp) \bigsqcup \spec \rhat[[u,v]]/(uv-\fancyp).
      \end{equation*}

\item When $(a,b) = \left(\dfrac{p-1}{2},0\right)$, after factorising
      \begin{equation*}
        \left( x^{(p-1)/2} - \fancyp^{(p^2-1)/2} \right) = 
        \prod_{\zeta \in \mu_{(p-1)/2}(\rhat)} \left(x - \zeta \fancyp^{p+1} \right)
      \end{equation*}
      we see that 
      \begin{equation*}
        \tcZ \times_{\cZ} \spec \widehat{\cO}_{\cZ,z} \cong 
        \bigsqcup_{\zeta \in \mu_{(p-1)/2}(\rhat)} \spec \rhat[[x,y]]/(x - \fancyp^{p+1}).
      \end{equation*}
      There are $(p-1)/2$ pre-images of this point in the normalisation. 

\item For $(a,b) = \left(\dfrac{p-1}{3},0 \right)$ similarly as above
      \begin{equation*}
        \tcZ \times_{\cZ} \spec \widehat{\cO}_{\cZ,z} \cong 
        \bigsqcup_{\zeta \in \mu_{(p-1)/3}(\rhat)} \spec \rhat[[x,y]]/(x - \fancyp^{3(p+1)/2}).
      \end{equation*}
      There are $(p-1)/3$ pre-images of this point in the normalisation. 

\item Next when $(a,b) = \left(\dfrac{p-1}{2},p-1 \right)$
      \begin{equation*}
        \left(x^{(p-1)/2}y^{p-1} - \fancyp^{(p^2-1)/2}\right) = 
        \prod_{\zeta \in \mu_{(p-1)/2} (\rhat)} (xy^2 - \zeta \fancyp^{p+1}).
      \end{equation*}
      Let $u = \fancyp^{(p+1)/2} / y$ then the normalisation of $\rhat[[x,y]]/(xy^2 - \fancyp^{p+1})$
      is 
      \begin{equation*}
        \rhat[[x,y,u]]/(x - u^2, uy - \fancyp^{(p+1)/2}) \cong \rhat[[u,y]]/(uy - \fancyp^{(p+1)/2}).
      \end{equation*}
      Hence
      \begin{equation*}
        \tcZ \times_{\cZ} \spec \widehat{\cO}_{\cZ,z} \cong 
        \bigsqcup_{\zeta \in \mu_{(p-1)/2}(\rhat)} \spec \rhat[[u,y]]/(uy - \fancyp^{(p+1)/2}).
      \end{equation*}
      Here $x = u^2$ and there are $(p-1)/2$ pre-images of this point in the normalisation.  

\item If $(a,b) = (\frac{p-1}{3},p-1)$ by an analogous computation as the previous case 
      we have 
      \begin{equation*}
        \tcZ \times_{\cZ} \spec \widehat{\cO}_{\cZ,z} \cong 
        \bigsqcup_{\zeta \in \mu_{(p-1)/3}(\rhat)} \spec \rhat[[u,y]]/(uy - \fancyp^{(p+1)/2}),
      \end{equation*}
      where $x=u^3$. There are $(p-1)/3$ pre-images of this point in the normalisation. 
\end{itemize}

From this and the calculations of Edixhoven \cite[Section 2]{MR1056773} it is clear that the special fibers of 
$\tcZ$ are given by Figure~\ref{fig:specialfib1}. The number adjacent to each component is its geometric
genus.
\begin{figure}[!htb]
  \centering
  \includegraphics[scale=0.5]{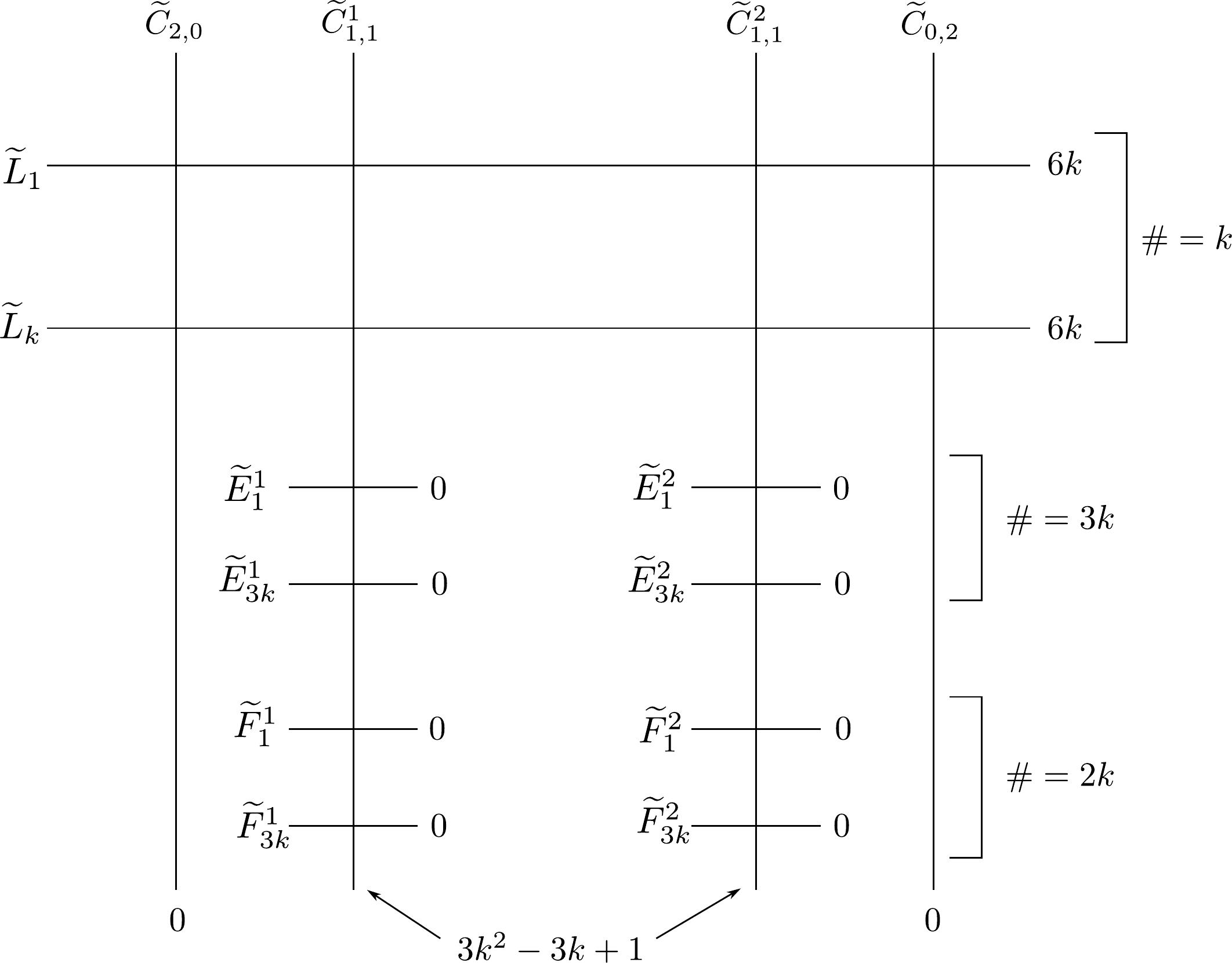}
  \caption{$\tcZ_{\prid_i}$ when $p = 12k+1$.}
  \label{fig:specialfib1}
\end{figure}
\newcommand{\wt}{\widetilde}

The map $\tcZ_{\prid_i} \to \cZ_{\prid_i}$  is described as follows. 
\begin{enumerate}
\item The components $\wt{C}_{2,0}$ and $\wt{C}_{0,2}$ map to $C_{2,0}$ and $C_{0,2}$ isomorphically.
\item The component $\wt{L}_i$ maps to $L_i$ with degree $p+1$. Over the points of intersection of $L_i$ with 
      $C_{2,0}$ and $C_{0,2}$ the map is totally ramified; at the points of intersection
      of $\wt{L}_i$ with $\wt{C}_{1,1}^1$ and $\wt{C}_{1,1}^2$ the ramification index is $(p-1)/2$.
\item The components $\wt{E}^j_i$ map isomorphically to $E$ and $\wt{F}^j_i$ map isomorphically to $F$.
\item On the other hand, $\wt{C}_{1,1}^j$ maps to $C_{1,1}$ with degree $(p-1)/2$. At the points of intersection 
      with $\wt{L}_i$ the map is totally ramified. At the points of intersection with
      $\wt{E}_i^j$, the ramification index is $2$ and at the points of intersection with
      $\wt{F}_i^j$ the ramification index is 3.    
\end{enumerate}

Only possible points where $\tcZ$ can fail to be regular are the singular points of the
special fibers. In the following table we list the local equation of all the double points
of the special fiber over $\prid_i$.
\begin{equation*}
  \renewcommand{\arraystretch}{1.2}
  \begin{array}{c|c|c|c|c}
    & \wt{C}_{2,0} & \wt{C}_{1,1}^1 & \wt{C}_{1,1}^2 & \wt{C}_{0,2} \\ \hline
    \wt{L}_i & ux - \fancyp^{(p-1)/2} & uv - \fancyp & uv - \fancyp & ux - \fancyp^{(p-1)/2}\\ \hline
    \wt{E}_i^1 & & uy - \fancyp^{(p+1)/2} &  & \\ \hline
    \wt{E}_i^2 & & & uy - \fancyp^{(p+1)/2}  & \\ \hline
    \wt{F}_i^1 & & uy - \fancyp^{(p+1)/2} &  & \\ \hline
    \wt{F}_i^2 & & & uy - \fancyp^{(p+1)/2}  & \\ 
  \end{array}
\end{equation*}

Hence the singular points are: the intersections of $\wt{L}_i$ with $\wt{C}_{2,0}$ and $\wt{C}_{0,2}$,
which we call $\alpha_i$ and $\prid_i$ respectively; the intersection points of $\wt{E}_i^j$ and 
$\wt{F}_i^j$ with $\wt{C}_{1,1}^j$ which we call $\sigma_i^j$ and $\tau_i^j$ respectively. At the other points
the local rings are regular. 

We resolve the singularities by successive blow ups as in Liu~\cite{MR1917232}. Let $\cW/\cO_K$ be the de-singularisation. 
The pre-image of $\alpha_i$ in $\cW$ is a tail consisting of projective lines, $A_{1,i}, \ldots, 
A_{(p-1)/2 - 1,i}$, where $A_{1,i}$ meets $\wt{C}_{2,0}$ and $A_{2,i}$; $A_{j,i}$ meets $A_{j-1,i}$
and $A_{j+1,i}$ for $1<j<(p-1)/2-1$; $A_{(p-1)/2 - 1,i}$ meets $A_{(p-1)/2 - 2,i}$ and $\wt{L}_i$.

Similarly, the pre-image of $\prid_i$ is a tail of projective lines $B_{1,i}, \ldots, 
B_{(p-1)/2 - 1,i}$ with analogous intersections.

The pre-image of $\sigma_i^j$ is again a union of $(p+1)/2-1 = (p-1)/2$ projective lines
$S_{1,i}^j, \ldots, S_{(p-1)/2,i}^j$ each meeting its successor, $S_{1,i}^j$ intersects $\wt{E}_i^j$
and $S_{(p-1)/2,i}^j$ intersects $\wt{C}_{1,1}^j$.

Finally the pre-image of $\tau_i^j$ is again a union of $(p+1)/2-1 = (p-1)/2$ projective lines
$T_{1,i}^j, \ldots, T_{(p-1)/2,i}^j$ each meeting its successor, $T_{1,i}^j$ intersects $\wt{F}_i^j$
and $T_{(p-1)/2,i}^j$ intersects $\wt{C}_{1,1}^j$.

The model $\cW/\cO_K$ is regular and semi-stable. However $\wt{E}_i^j$ are projective lines with self-intersection 
$-1$ and can be blown-down, and similarly we can blow down $\wt{F}_i^j$. Successively we can blow-down all 
$S_{l,i}^j$ and $T_{l,i}^j$. 

Let $\cX_0(p^2)/\cO_K$ be the resulting arithmetic surface; then it is the minimal regular model of 
$X_0(p^2)/K$ if $k>1$ (see figure \ref{fig:semistable1}). 
We deal with the case $p=13$ separately in the appendix.
\begin{figure}[!htb]
  \centering
  \includegraphics[scale=0.6]{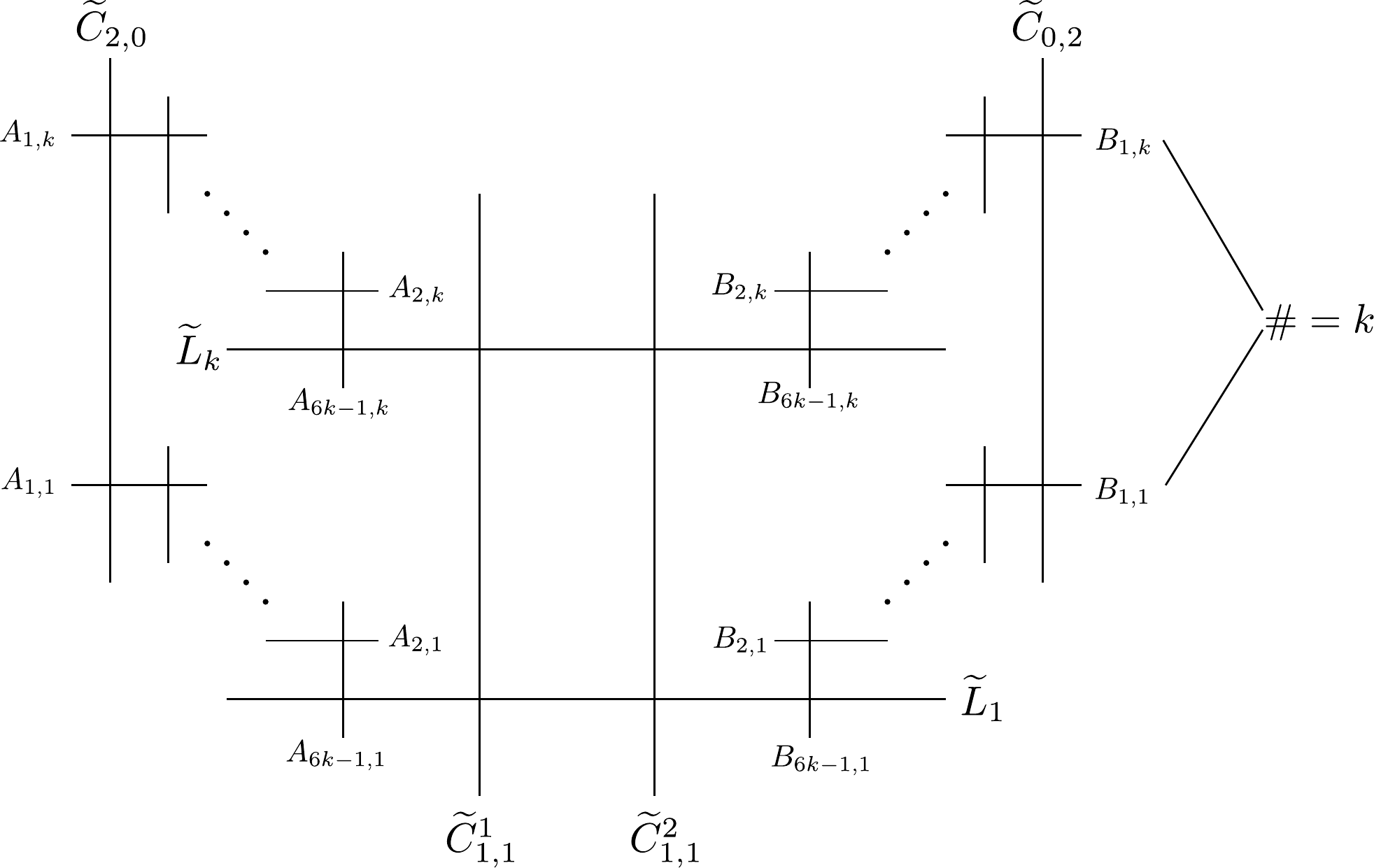}
  \caption{Special fiber $\cX_0(p^2)_{\prid_i}$ when $p=12k+1$, $k>1$.}
  \label{fig:semistable1}
\end{figure}

\subsection{Case 2: $p = 12k+5$} \label{subsec:5mod12}
Here we encounter the complete local rings
\begin{equation*}
  \widehat{\cO}_{\cZ,z} \cong \frac{\rhat[[x,y]]}{(x^a y^b - \fancyp^{(p^2-1)/2})},
\end{equation*}
where at the double points of the special fiber $\cZ_{\prid_i}$
\begin{equation*}
  (a,b) \in \left\{(p+1,1), (p+1,p-1), \left(\frac{p-1}{2}, p-1\right), 
                   \left(\frac{p+1}{3}, p-1\right), \left(\frac{p+1}{3},1\right) \right\}
\end{equation*}
and at smooth points of the special fiber 
\begin{equation*}
  (a,b) \in \left\{1,p+1,p-1,\frac{p-1}{2}, \frac{p+1}{3} \right\} \times \{ 0 \}.
\end{equation*}

We only discuss the cases that have not been dealt with already. The new cases that we encounter 
presently are $(a,b) \in \left\{ \left(\dfrac{p+1}{3},0\right), \left(\dfrac{p+1}{3},1\right), 
\left(\dfrac{p+1}{3},p-1\right) \right\}$.

\begin{itemize} \itemsep10pt
\item When $(a,b) = \left(\dfrac{p+1}{3},0\right)$ we have
      \begin{equation*}
        \left(x^{(p+1)/3} - \fancyp^{(p^2-1)/2} \right) = \prod_{\zeta \in \mu_{(p+1)/3}(\rhat)} 
        \left(x - \zeta\fancyp^{3(p-1)/2} \right).
      \end{equation*}
      Hence
      \begin{equation*} \label{eq:case2b0}
        \tcZ \times_{\cZ} \spec \widehat{\cO}_{\cZ,z} \cong 
        \bigsqcup_{\zeta \in \mu_{(p+1)/3}(\rhat)} \spec \rhat[[x,y]]/(x - \fancyp^{3(p-1)/2}).
      \end{equation*}
      There are $(p+1)/3$ pre-images of this point in the normalisation.

\item When $(a,b) = \left(\dfrac{p+1}{3},1\right)$ we have to calculate the normalisation of 
      $\rhat[[x,y]]/(x^{(p+1)/3}y - \fancyp^{(p^2-1)/2})$. We can in this case take 
      $u = \fancyp^{3(p-1)/2}/x$, then the normalisation is 
      \begin{equation*}
        \rhat[[x,y,u]]/(y-u^{(p+1)/3}, ux - \fancyp^{3(p-1)/2}) \cong 
        \rhat[[x,u]]/(ux - \fancyp^{3(p-1)/2}).
      \end{equation*}
      Thus
      \begin{equation*} \label{eq:case2b1}
        \tcZ \times_{\cZ} \spec \widehat{\cO}_{\cZ,z} \cong 
        \spec \rhat[[x,u]]/(ux - \fancyp^{3(p-1)/2})
      \end{equation*}
      with $y=u^{(p+1)/3}$. There is only one pre-image of this point in the normalisation.

\item Finally when $(a,b) = \left(\dfrac{p+1}{3},p-1\right)$, we can factorise
      \begin{equation*}
        \left(x^{(p+1)/3} y^{p-1} - \fancyp^{(p^2-1)/2}\right) = 
        \left(x^{(p+1)/6} y^{(p-1)/2} - \fancyp^{(p^2-1)/4}\right)
                         \left(x^{(p+1)/6} y^{(p-1)/2} + \fancyp^{(p^2-1)/4}\right).
      \end{equation*}
      Consider
      \begin{equation*}
        \rhat[[x,y]]/(x^{(p+1)/6} y^{(p-1)/2} -\fancyp^{(p^2-1)/4})
      \end{equation*}
      taking $u = \fancyp^{(p+1)/2}/y$ we get an integral extension
      \begin{equation*}
        \rhat[[x,y,u]]/(x^{(p+1)/6} - u^{(p-1)/2},uy - \fancyp^{(p+1)/2}),
      \end{equation*} 
      again taking $v = x/u^2$ we get the integral extension
      \begin{equation*}
        \rhat[[y, u, v]]/(v^{(p+1)/6} - u^{(p-5)/6}, uy - \fancyp^{(p+1)/2}).
      \end{equation*}
      Taking $t = u/v$ we get the integral extension
      \begin{equation*}
        \rhat[[y,v,t]]/(v-t^{(p-5)/6}, yvt - \fancyp^{(p+1)/2}) \cong
        \rhat[[y,t]]/(yt^{(p+1)/6} - \fancyp^{(p+1)/2}).
      \end{equation*}
      Finally we can take $s = \fancyp^3/t$ to get the extension
      \begin{equation*}
        \rhat[[y,s,t]]/(st- \fancyp^3, y - s^{(p+1)/6}) \cong
        \rhat[[s,t]]/(st- \fancyp^3)
      \end{equation*}
      which is normal but not regular. We have
      \begin{equation*} \label{eq:case2bp-1}
        \tcZ \times_{\cZ} \spec \widehat{\cO}_{\cZ,z} \cong 
        \spec \rhat[[s,t]]/(st - \fancyp^3) \ \bigsqcup \ \spec \rhat[[s,t]]/(st - \fancyp^3).
      \end{equation*}
      There are thus 2 pre-images of this point in the normalization.
\end{itemize}      

From these calculations  and those of Edixhoven, section 2 it is clear that the special fiber
$\tcZ_{\prid_i}$ is given by Figure~\ref{fig:specialfib5}.
\begin{figure}[!htb]
  \centering
  \includegraphics[scale=0.5]{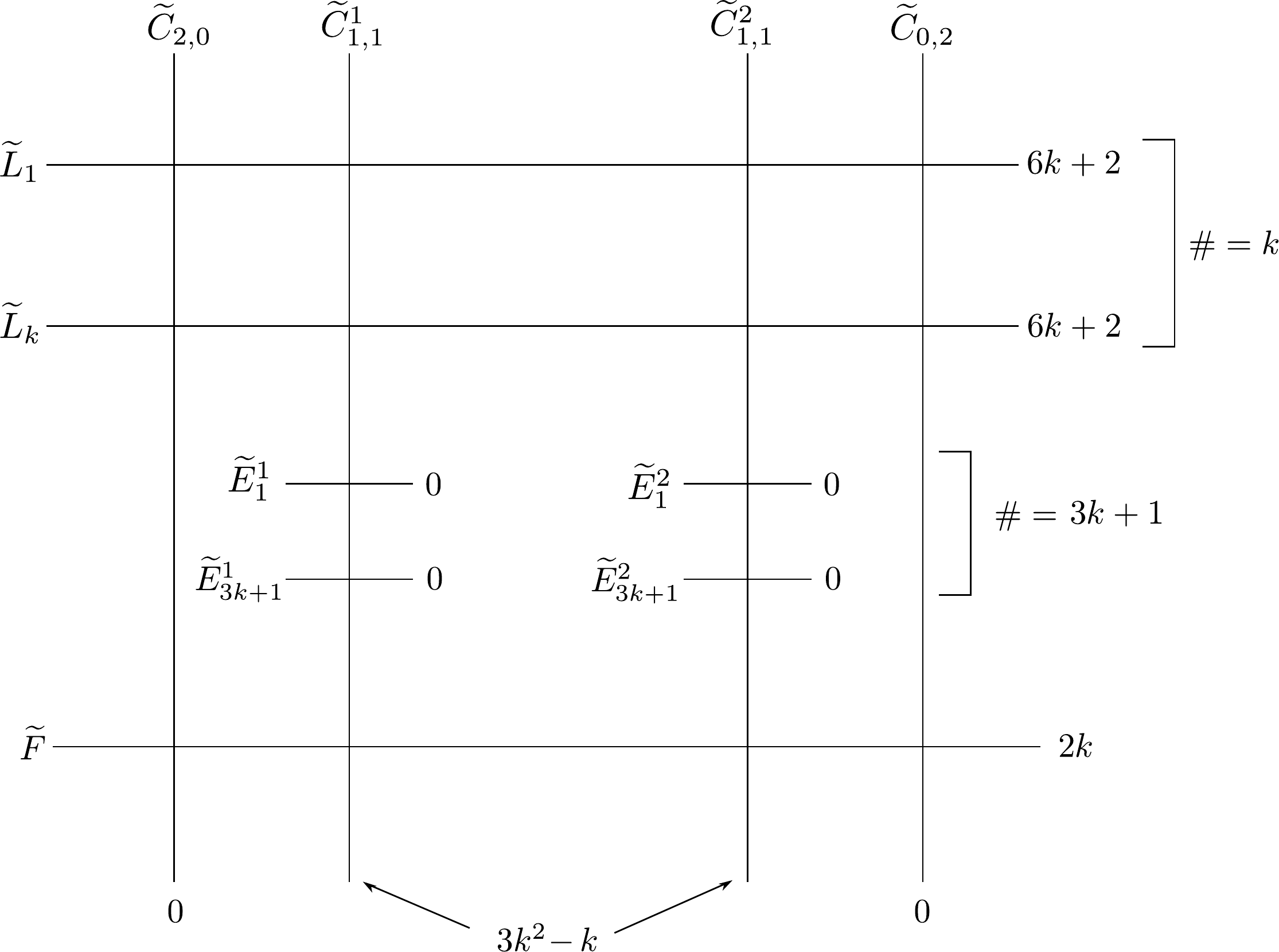}
  \caption{$\tcZ_{\prid_i}$ when $p = 12k+5$.}
  \label{fig:specialfib5}
\end{figure}

The map $\tcZ_{\prid_i} \to \cZ_{\prid_i}$ is described as follows:
\begin{enumerate}
\item The components $\wt{C}_{2,0}$ and $\wt{C}_{0,2}$ map to $C_{2,0}$ and $C_{0,2}$ isomorphically.
\item The component $\wt{L}_i$ maps to $L_i$ with degree $p+1$. Over the points of intersection of $L_i$ with 
      $\wt{C}_{2,0}$ and $\wt{C}_{0,2}$ the map is totally ramified; at the points of intersection
      of $\wt{L}_i$ with $\wt{C}_{1,1}^1$ and $\wt{C}_{1,1}^2$ the ramification index is $(p+1)/2$.
\item The component $\wt{E}_i^j$ maps isomorphically to $E$.
\item The component $\wt{F}$ maps to $F$ with degree $(p+1)/3$, the point on $F$ at which it 
      intersects $C_{2,0}$ has only one pre-image where we have total ramification and 
      same for the point where it intersects $C_{0,2}$. The point where $F$ intersects
      $C_{1,1}$ has 2 pre-images and the ramification index is $(p+1)/6$ at both of these points.
\item The component $\wt{C}_{1,1}^j$ maps to $C_{1,1}$ with degree $(p-1)/2$. At the points of intersection 
      with $\wt{L}_i$ the map is totally ramified. At the points of intersection with
      $\wt{E}_i^j$, the ramification index is $2$ and at the points of intersection with
      $\wt{F}$ the map is again totally ramified.
\end{enumerate}

The local equations at the double points of the special fiber are listed in the following table.
\begin{equation*}
  \renewcommand{\arraystretch}{1.2}
  \begin{array}{c|c|c|c|c}
    & \wt{C}_{2,0} & \wt{C}_{1,1}^1 & \wt{C}_{1,1}^2 & \wt{C}_{0,2} \\ \hline
    \wt{L}_i & ux - \fancyp^{(p-1)/2} & uv - \fancyp & uv - \fancyp & ux - \fancyp^{(p-1)/2}\\ \hline
    \wt{E}_i^1 & & uy - \fancyp^{(p+1)/2} &  & \\ \hline
    \wt{E}_i^2 & & & uy - \fancyp^{(p+1)/2}  & \\ \hline
    \wt{F} & ux - \fancyp^{3(p-1)/2} & st - \fancyp^3 & st - \fancyp^3  & ux - \fancyp^{3(p-1)/2}\\ 
  \end{array}
\end{equation*}

When $k>0$ the special fibers of the minimal regular model $\cX_0(p^2)/\cO_K$ are shown in
Figure~\ref{fig:semistable5}. $X_0(25)$ has genus 0.
\begin{figure}[!htb]
  \centering
  \includegraphics[scale=0.6]{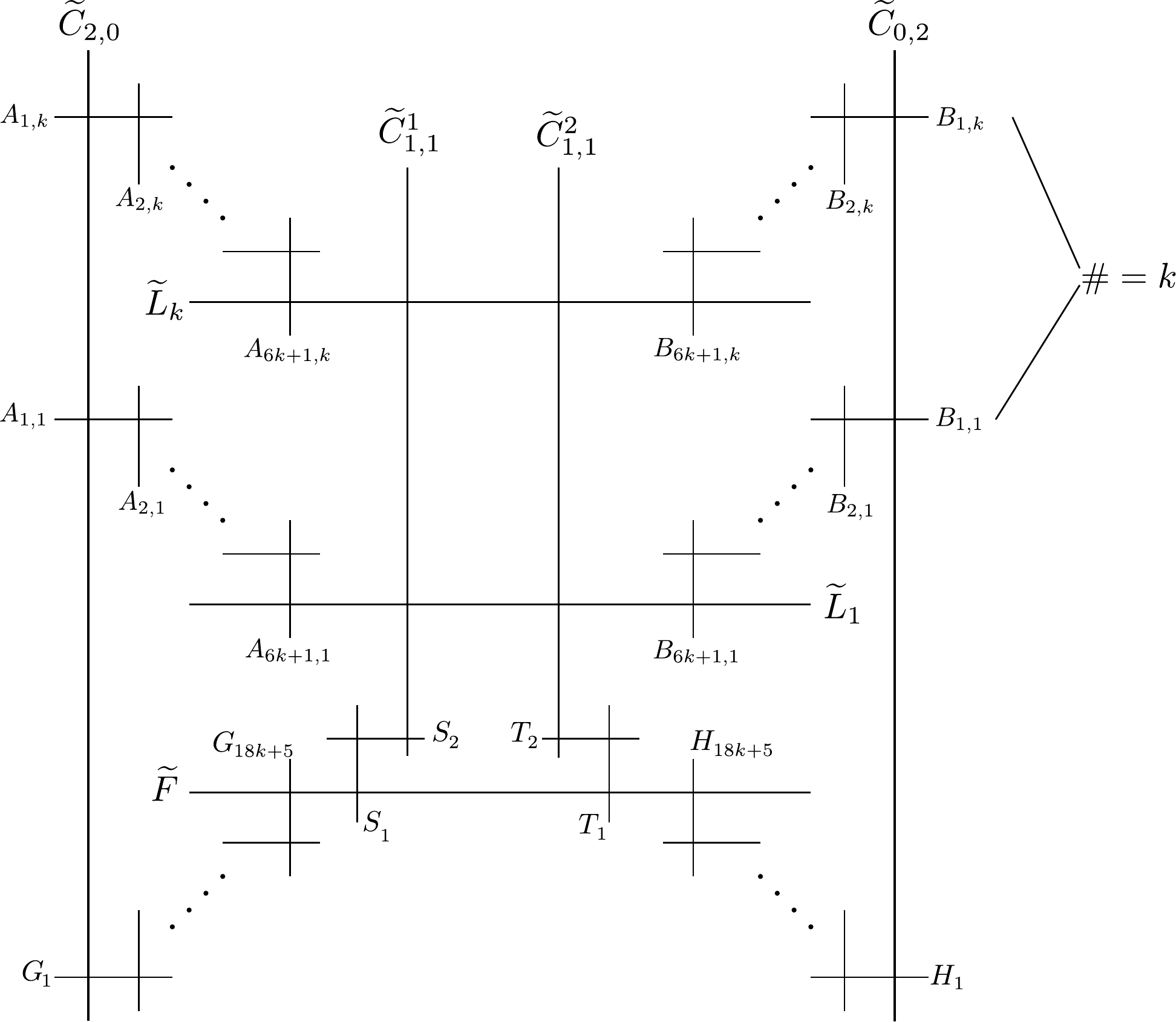}
  \caption{$\cX_0(p^2)_{\prid_i}$ when $p=12k+5$, $k>0$.}
  \label{fig:semistable5}
\end{figure}

\subsection{Case 3: $p=12k+7$} \label{subsec:7mod12}
Now we encounter the complete local rings
\begin{equation*}
  \widehat{\cO}_{\cZ,z} \cong \frac{\rhat[[x,y]]}{(x^a y^b - \fancyp^{(p^2-1)/2})},
\end{equation*}
where at the double points of the special fiber
\begin{equation*}
  (a,b) \in \left\{(p+1,1), (p+1,p-1), \left(\frac{p+1}{2}, p-1\right), \left(\frac{p+1}{2},1\right),
  \left(\frac{p-1}{3}, p-1\right) \right\}
\end{equation*}
and at smooth points of $\cZ_{\prid_i}$
\begin{equation*}
  (a,b) \in \left\{1,p+1,p-1,\frac{p+1}{2}, \frac{p-1}{3} \right\} \times \{ 0 \}.
\end{equation*}
The new cases are $(a,b) \in \left\{\left(\dfrac{p+1}{2},0\right), \left(\dfrac{p+1}{2},1\right), 
\left(\dfrac{p+1}{2},p-1\right) \right\}$.

\begin{itemize} \itemsep10pt
\item When $(a,b) = \left(\dfrac{p+1}{2},0\right)$ we have
      \begin{equation*}
        \left(x^{(p+1)/2} - \fancyp^{(p^2-1)/2} \right) = \prod_{\zeta \in \mu_{(p+1)/2}(\rhat)} 
        \left(x - \zeta\fancyp^{p-1}\right).
      \end{equation*}
      Hence
      \begin{equation*} \label{eq:case3b0}
        \tcZ \times_{\cZ} \spec \widehat{\cO}_{\cZ,z} \cong 
        \bigsqcup_{\zeta \in \mu_{(p+1)/2}(\rhat)} \spec \rhat[[x,y]]/(x - \fancyp^{p-1}).
      \end{equation*}

\item When $(a,b) = \left(\dfrac{p+1}{2},1\right)$ we have to calculate the normalisation of 
      $\rhat[[x,y]]/(x^{(p+1)/2}y - \fancyp^{(p^2-1)/2})$. We can in this case take 
      $u = \fancyp^{p-1}/x$, then the normalisation is 
      \begin{equation*}
        \rhat[[x,y,u]]/(y-u^{(p+1)/2}, ux - \fancyp^{p-1}) \cong 
        \rhat[[x,u]]/(ux - \fancyp^{p-1}).
      \end{equation*}
      Thus
      \begin{equation*} \label{eq:case3b1}
        \tcZ \times_{\cZ} \spec \widehat{\cO}_{\cZ,z} \cong 
        \spec \rhat[[x,u]]/(ux - \fancyp^{p-1})
      \end{equation*}
      with $y=u^{(p+1)/2}$. There is only one pre-image of this point in the normalisation.

\item Finally, when $(a,b) = \left(\dfrac{p+1}{2},p-1\right)$ we can factorise
      \begin{equation*}
        \left(x^{(p+1)/2} y^{p-1} - \fancyp^{(p^2-1)/2}\right) = 
        \left(x^{(p+1)/4} y^{(p-1)/2} - \fancyp^{(p^2-1)/4}\right)
                         \left(x^{(p+1)/4} y^{(p-1)/2} + \fancyp^{(p^2-1)/4}\right).
      \end{equation*}
      Consider
      \begin{equation*}
        \rhat[[x,y]]/(x^{(p+1)/4} y^{(p-1)/2} -\fancyp^{(p^2-1)/4})
      \end{equation*}
      taking $u = \fancyp^{(p+1)/2}/y$ we get an integral extension
      \begin{equation*}
        \rhat[[x,y,u]]/(x^{(p+1)/4} - u^{(p-1)/2},uy - \fancyp^{(p+1)/2}),
      \end{equation*} 
      again taking $v = x/u$ we get the integral extension
      \begin{equation*}
        \rhat[[y, u, v]]/(v^{(p+1)/4} - u^{(p-3)/4}, uy - \fancyp^{(p+1)/2}).
      \end{equation*}
      Taking $t = u/v$ we get the integral extension
      \begin{equation*}
        \rhat[[y,v,t]]/(v-t^{(p-3)/4}, yvt - \fancyp^{(p+1)/2}) \cong
        \rhat[[y,t]]/(yt^{(p+1)/4} - \fancyp^{(p+1)/2}).
      \end{equation*}
      Finally we can take $s = \fancyp^2/t$ to get the extension
      \begin{equation*}
        \rhat[[y,s,t]]/(st- \fancyp^2, y - s^{(p+1)/4}) \cong
        \rhat[[s,t]]/(st- \fancyp^2)
      \end{equation*}
      which is normal but not regular. We thus have
      \begin{equation*} \label{eq:case3bp-1}
        \tcZ \times_{\cZ} \spec \widehat{\cO}_{\cZ,z} \cong 
        \spec \rhat[[s,t]]/(st - \fancyp^2) \ \bigsqcup \ \spec \rhat[[s,t]]/(st - \fancyp^2).
      \end{equation*}
      There are thus 2 pre-images of this point in the normalisation.
\end{itemize}      
      
These together give the special fibers of $\tcZ/\cO_K$ which we describe below.
\begin{center}
\begin{figure}[!htb]
  \includegraphics[scale=0.5]{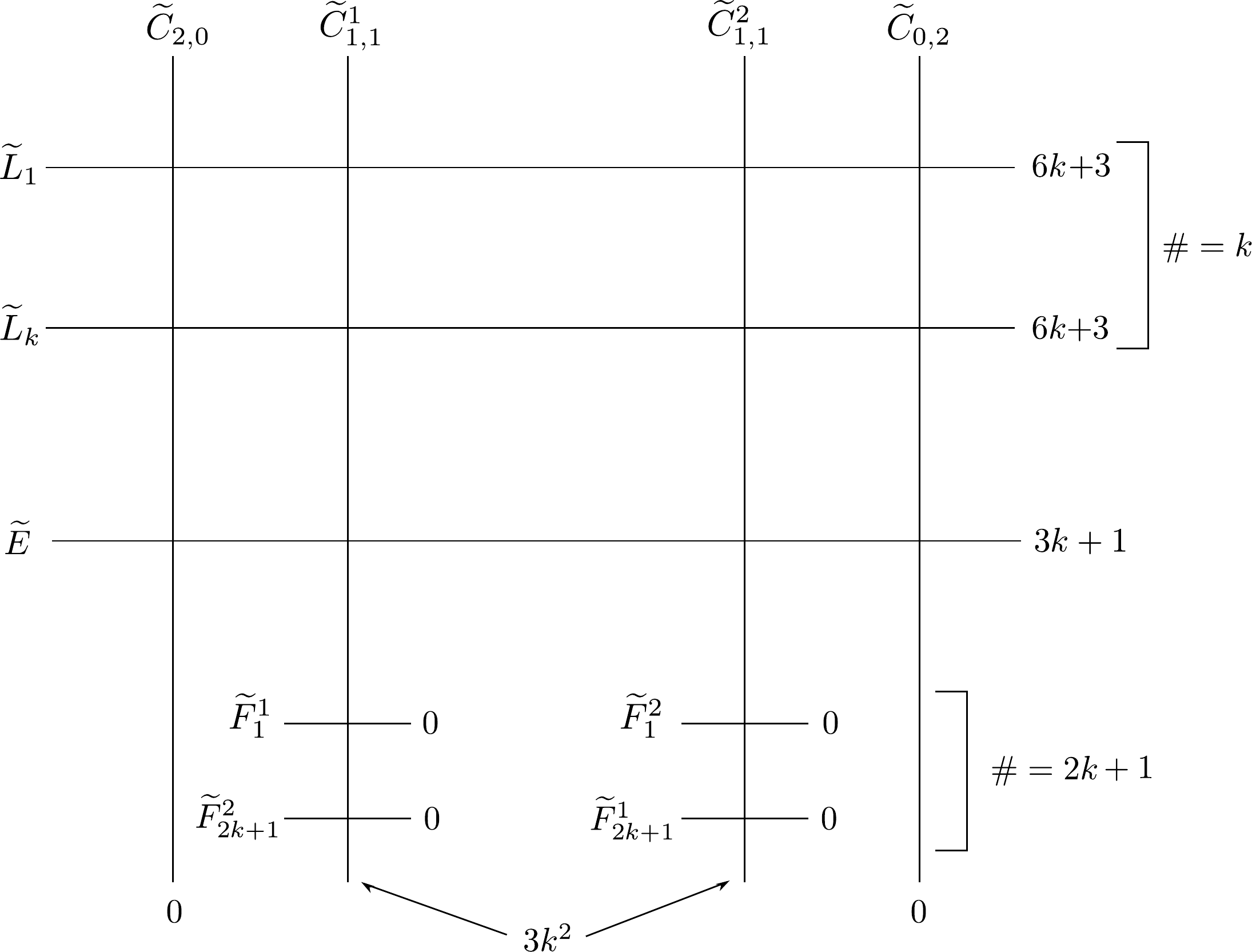}
  \caption{$\tcZ_{\prid_i}$ when $p = 12k+7$.}
\end{figure}
\end{center}

The map of the special fibers $\tcZ_{\prid_i} \to 
\cZ_{\prid_i}$ is described as follows.
\begin{enumerate}
\item The components $\wt{C}_{2,0}$ and $\wt{C}_{0,2}$ map to $C_{2,0}$ and $C_{0,2}$ isomorphically.
\item The component $\wt{L}_i$ maps to $L_i$ with degree $p+1$. At the points of intersection with 
      $\wt{C}_{2,0}$ and $\wt{C}_{0,2}$ the map is totally ramified; at the points of intersection
      with $\wt{C}_{1,1}^1$ and $\wt{C}_{1,1}^2$ the ramification index is $(p+1)/2$.
\item The component $\wt{F}_i^j$ maps isomorphically to $F$.
\item The component $\wt{E}$ maps to $E$ with degree $(p+1)/2$, the point on $E$ at which it 
      intersects $C_{2,0}$ has only one pre-image where we have total ramification and 
      same for the point where it intersects $C_{0,2}$. The point where $E$ intersects
      $C_{1,1}$ has 2 pre-images and the ramification index is $(p+1)/4$ at both of these points.
\item The components $\wt{C}_{1,1}^j$ map to $C_{1,1}$ with degree $(p-1)/2$. At the points of intersection 
      with $\wt{L}_i$ the map is totally ramified. At the points of intersection with
      $\wt{F}_i^j$, the ramification index is $3$ and at the points of intersection with
      $\wt{E}$ the map is again totally ramified.
\end{enumerate}

Below we give the local equations at the nodes of the special fiber.
\begin{equation*}
  \renewcommand{\arraystretch}{1.2}
  \begin{array}{c|c|c|c|c}
    & \wt{C}_{2,0} & \wt{C}_{1,1}^1 & \wt{C}_{1,1}^2 & \wt{C}_{0,2} \\ \hline
    \wt{L}_i & ux - \fancyp^{(p-1)/2} & uv - \fancyp & uv - \fancyp & ux - \fancyp^{(p-1)/2}\\ \hline
    \wt{E} & ux - \fancyp^{p-1} & st-\fancyp^2 & st-\fancyp^2 & ux - \fancyp^{p-1} \\ \hline
    \wt{F}_i^1 & & uy - \fancyp^{(p+1)/2} &  & \\ \hline
    \wt{F}_i^2 & & & uy - \fancyp^{(p+1)/2}  & \\ 
  \end{array}
\end{equation*}

When $k>0$ the special fiber of the minimal regular model $\cX_0(p^2)/\cO_K$ is depicted by 
Figure~\ref{fig:semistable7}. We deal with the case $p=7$ separately in the appendix.
\begin{figure}[!htb]
  \centering
  \includegraphics[scale=0.6]{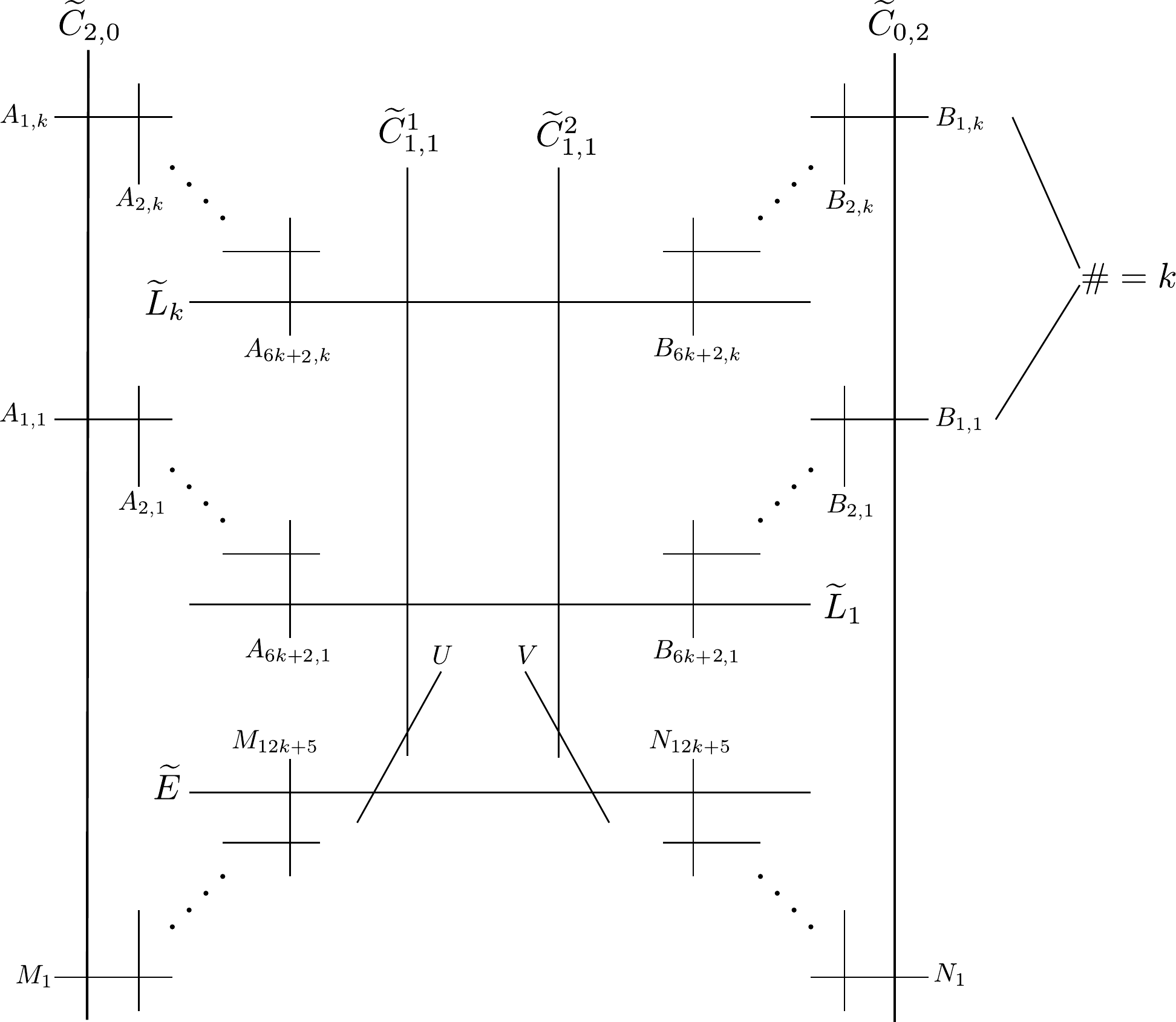}
  \caption{Special fibers of $\cX_0(p^2)/\cO_K$ when $p=12k+7$, $k>0$.}
  \label{fig:semistable7}
\end{figure}

\subsection{Case 4: $p=12k+11$} \label{subsec:11mod12}
We encounter the complete local rings
\begin{equation*}
  \widehat{\cO}_{\cZ,z} \cong \frac{\rhat[[x,y]]}{(x^a y^b - \fancyp^{(p^2-1)/2})},
\end{equation*}
where at the double points of the special fiber
\begin{equation*}
  (a,b) \in \left\{p+1, \frac{p+1}{2}, \frac{p+1}{3} \right\} \times \{1, p-1\}
\end{equation*}
and at smooth points of the special fiber
\begin{equation*}
  (a,b) \in \left\{1,p+1,p-1,\frac{p+1}{2}, \frac{p+1}{3} \right\} \times \{ 0 \}.
\end{equation*}

Since all these cases have already been dealt with we just draw the special fiber below.
\begin{figure}[!htb]
  \centering
  \includegraphics[scale=0.5]{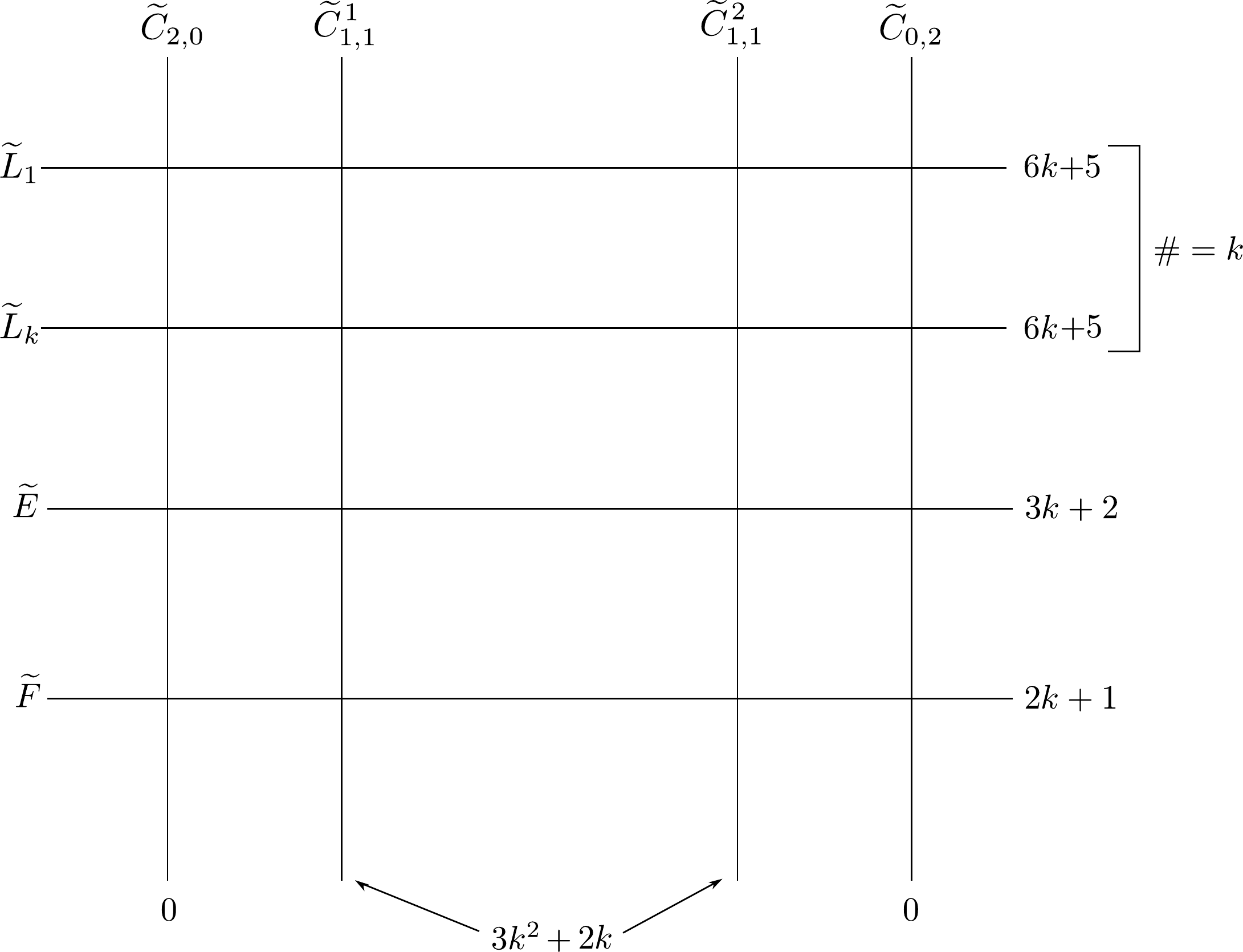}
  \caption{$\tcZ_{\prid_i}$ when $p = 12k+11$.}
\end{figure}

The map of the special fibers $\tcZ_{\prid_i} \to \cZ_{\prid_i}$ is described as follows.
\begin{enumerate}
\item The components $\wt{C}_{2,0}$ and $\wt{C}_{0,2}$ map to $C_{2,0}$ and $C_{0,2}$ isomorphically.
\item The component $\wt{L}_i$ maps to $L_i$ with degree $p+1$. At the points of intersection with 
      $\wt{C}_{2,0}$ and $\wt{C}_{0,2}$ the map is totally ramified; at the points of intersection
      with $\wt{C}_{1,1}^1$ and $\wt{C}_{1,1}^2$ the ramification index is $(p+1)/2$.
\item The component $\wt{E}$ maps to $E$ with degree $(p+1)/2$. The point on $E$ at which it 
      intersects $C_{2,0}$ has only one pre-image where we have total ramification and 
      same for the point where it intersects $C_{0,2}$. The point where $E$ intersects
      $C_{1,1}$ has 2 pre-images and the ramification index is $(p+1)/4$ at both of these points.
\item The component $\wt{F}$ maps to $F$ with degree $(p+1)/3$, the point on $F$ at which it 
      intersects $C_{2,0}$ has only one pre-image where we have total ramification and 
      same for the point where it intersects $C_{0,2}$. The point where $F$ intersects
      $C_{1,1}$ has 2 pre-images and the ramification index is $(p+1)/6$ at both of these points.      
\item The components $\wt{C}_{1,1}^j$ map to $C_{1,1}$ with degree $(p-1)/2$. At the points of intersection 
      with $\wt{L}_i$, $\wt{E}$ and $\wt{F}$ the map is totally ramified. 
\end{enumerate}

The local equations at the nodes of the special fiber are the following.
\begin{equation*}
  \renewcommand{\arraystretch}{1.2}
  \begin{array}{c|c|c|c|c}
    & \wt{C}_{2,0} & \wt{C}_{1,1}^1 & \wt{C}_{1,1}^2 & \wt{C}_{0,2} \\ \hline
    \wt{L}_i & ux - \fancyp^{(p-1)/2} & uv - \fancyp & uv - \fancyp & ux - \fancyp^{(p-1)/2}\\ \hline
    \wt{E} & ux - \fancyp^{p-1} & st-\fancyp^2 & st-\fancyp^2 & ux - \fancyp^{p-1} \\ \hline
    \wt{F} & ux - \fancyp^{3(p-1)/2} & st - \fancyp^3 & st - \fancyp^3  & ux - \fancyp^{3(p-1)/2}\\ 
  \end{array}
\end{equation*}

The special fibers of the minimal regular model $\cX_0(p^2)/\cO_K$ are shown in 
Figure~\ref{fig:semistable11}.
\begin{figure}[!htb]
  \centering
  \includegraphics[scale=0.6]{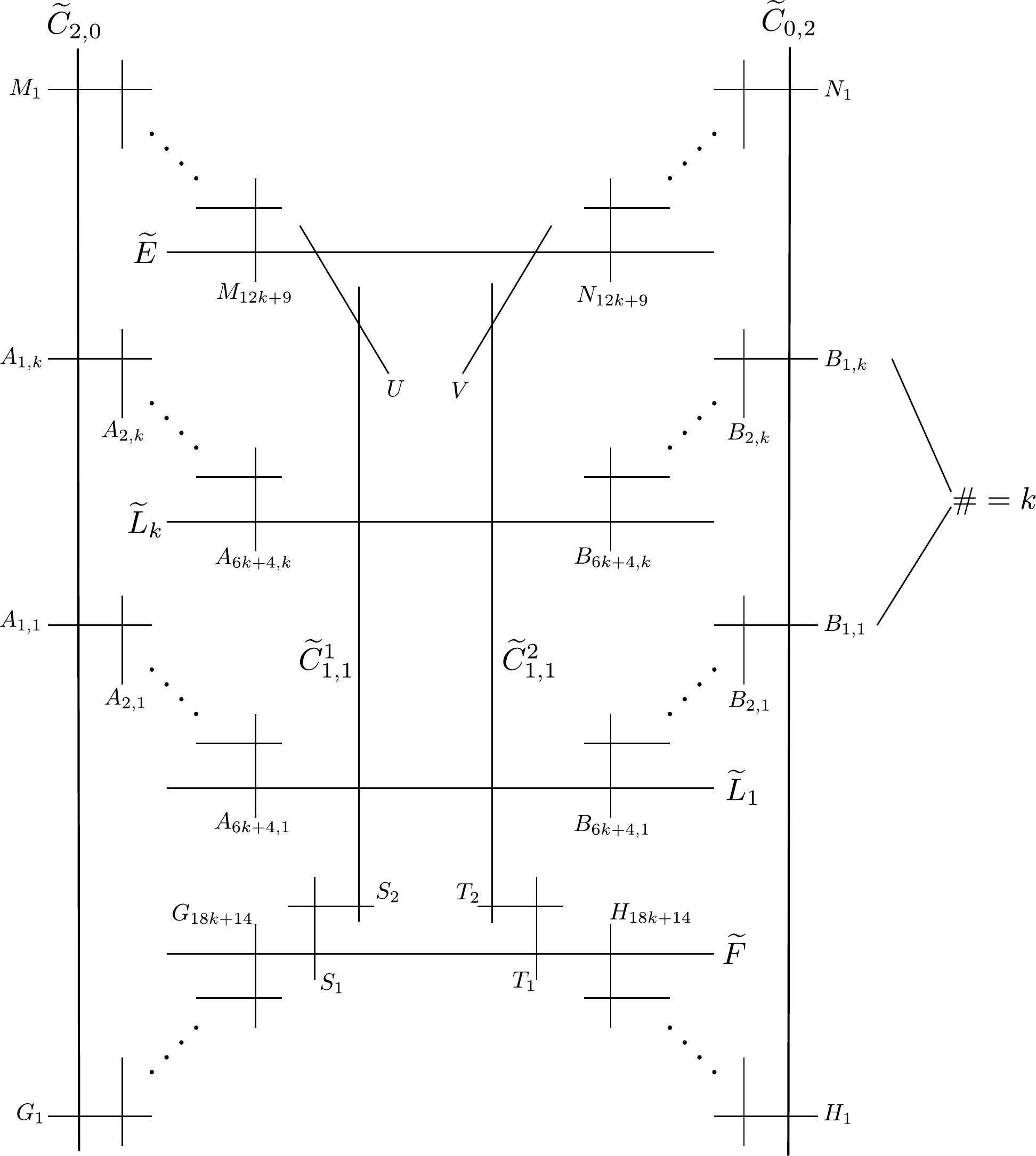}
  \caption{Special fibers of $\cX_0(p^2)/\cO_K$ when $p=12k+11$.}
  \label{fig:semistable11}
\end{figure}

\pagebreak

\section{Arakelov self-intersection of the canonical sheaf} \label{sec:self-inter}
\newcommand{\pichat}{\widehat{\pic}}
\newcommand{\cL}{\mathscr{L}}
\newcommand{\cM}{\mathscr{M}}
\newcommand{\arsq}[1]{\left( #1 \right)^{2,\mathrm{Ar}}}
\newcommand{\fK}{\mathcal{K}}
\newcommand{\can}{\fK_{\cX_0(p^2)}}
\newcommand{\ovcan}{\ov{\fK}_{\cX_0(p^2)}}

In this section we derive an asymptotic expression for the stable arithmetic self-intersection number 
$\selfinter$ of $X_0(p^2)$. For this we carry out a program similar to Abbes-Ullmo \cite{\au}
or Mayer \cite{\may} on the semistable models obtained in the previous section. As a corollary we obtain
an asymptotic expression for the stable Faltings height of $X_0(p^2)$. We start the section with a quick 
introduction to Arakelov intersection pairing. For a thorough account of Arakelov theory one should refer 
to \cite{\ara}, \cite{\fal} and \cite{\lang}. 

\subsection{Arakelov intersection pairing}
\label{Arakelovintersection}
Let $F$ be a number field and $R$ be its ring of integers. Let $\cX$ be an arithmetic surface over $\spec R$
(in the sense of Liu \cite[Chapter 8, Definition 3.14]{\liu}) with the map $f: \cX \to \spec R$. 
Let $X = \cX_{(0)}$ be the generic fiber which is a smooth irreducible projective curve over $F$.
In our case, this generic fiber is the compact Riemann surface $X_0(p^2)$. 
 For each 
embedding $\sigma: F \to \cc$ we get a smooth proper curve   $\cX_{\sigma}$  over $\cc$ by base change
\begin{equation*}
 \cX_{\sigma}:= X \times_{\spec F, \sigma} \spec \cc.
\end{equation*}
This is a smooth curve over $\cc$ via the second projection. The analytification of $ \cX_{\sigma}$ gives 
a compact Riemann surface of genus $g_{\sigma}$. In fact all these Riemann surfaces are isomorphic, so 
let $g_X = g_{\sigma}$, and assume $g_X > 1$. Arakelov introduced an intersection pairing for divisors and 
line bundles on arithmetic surfaces which is well suited for arithmetic geometry.

Any line bundle $\cL$ on $\cX$ induces a line bundle $\cL_{\sigma}$ on the Riemann surface $\cX_{\sigma}$.
A metrized line bundle $\ov{\cL} = (\cL, h)$ is a line bundle $\cL$ on $\cX$ along 
with a hermitian metric $h_{\sigma}$ on each $\cL_{\sigma}$. The isometry classes of metrized line bundles on 
$\cX$ form a group under tensor product: this is called the arithmetic Picard group and denoted by $\pichat(\cX)$. 
There is a symmetric bilinear pairing 
\begin{equation*}
  \pichat(\cX) \times \pichat(\cX) \to \rr
\end{equation*}
called the arithmetic intersection pairing, see for instance \cite{\fal} or \cite{\lang}.

Following Arakelov \cite{\ara} we can associate a unique hermitian metric on $\cL_{\sigma}$ for each $\sigma$ called 
the Arakelov metric. This turns $\cL$ into a metrized line bundle in a unique way. The Arakelov intersection of
two line bundles is the arithmetic intersection of those line bundles equipped with the Arakelov metric. 
Arakelov intersection pairing is thus a special case of arithmetic intersection pairing. 
We shall not describe the more general arithmetic intersection theory, but restrict ourselves to Arakelov
intersection. 

Let $\cL$ and $\cM$ be two line bundles on $\cX$. Let $\ov{\cL}$ and $\ov{\cM}$ be the corresponding unique 
metrized line bundles. The Arakelov intersection of $\cL$ and $\cM$ is given by
\begin{equation*}
  \ar{\ov\cL}{\ov\cM} = \langle \cL, \cM \rangle_{\fin} + 
                                     \sum_{\sigma: F \to \cc} \langle \cL_{\sigma}, \cM_{\sigma} \rangle.
\end{equation*}
The first summand is the algebraic or finite part whereas the second summand is the analytic or infinite 
part of the intersection. 

To describe the finite part, let us assume $l$ and $m$ are non-trivial global 
sections of $\cL$ and $\cM$ respectively such that the associated divisors do not have any common components. 
For each closed point $x \in \cX$, $l_x$ and $m_x$ can be thought of as elements of the local ring $\cO_{\cX, x}$ 
via a suitable trivialization. If $\cX^{(2)}$ is the set of closed points of $\cX$, 
(the number 2 here signifies the fact that a closed point is an algebraic cycle on $\cX$ of codimension 2), then 
\begin{equation*}
  \langle \cL, \cM \rangle_{\fin} := \sum_{x \in \cX^{(2)}} \log \# (\cO_{\cX,x}/ (l_x, m_x)).
\end{equation*}
This is closely related to the usual algebraic intersection product defined in Liu \cite[Chapter 9, 
Definition 1.15]{\liu}.

To describe the infinite part we need to introduce some notation. 
Let $X$ be any compact Riemann surface of genus $g_X \geq 1$. Note that the space $H^0(X, \Omega^1)$ of 
holomorphic differentials on $X$ has a natural inner product on it given by
$\displaystyle \langle \phi, \psi \rangle = \frac{i}{2} \int_{X} \phi \wedge \overline{\psi}.$
Choose an orthonormal basis $f_1, \ldots, f_{g_X}$ of $H^0(X, \Omega^1)$. 
The canonical volume form on $X$ is 
\begin{equation*}
  \mucan = \frac{i}{2g_X} \sum_{j=1}^{g} f_j \wedge \overline{f_j}.
\end{equation*}

\begin{definition}  \label{def:green}
  The canonical Green's function $\gcan$ on the compact Riemann surface $X$ is the 
  unique solution to the differential equation
  \begin{equation*}
    \partial_{z} \partial_{\ov z}\ \gcan (z,w)=i\pi(\mucan (z)-\de_w(z))
  \end{equation*}
  with the normalization condition
  $\displaystyle \int_{X} \gcan(z,w) \mucan(z) = 0$. 
  Here $\de_w(z)$ is the Dirac delta distribution. 
  \end{definition}

Assume that the divisors corresponding to $\cL_{\sigma}$ and $\cM_{\sigma}$ are given by
\begin{equation*}
  \Div(\cL_{\sigma}) = \sum_{\alpha} n_{\alpha, \sigma} P_{\alpha}^{\sigma}, \quad
  \text{and} \quad 
  \Div(\cM_{\sigma}) = \sum_{\prid} r_{\prid, \sigma} Q_{\prid}^{\sigma}
\end{equation*}
then 
\begin{equation*}
  \langle \cL_{\sigma}, \cM_{\sigma} \rangle :=
                              - \sum n_{\alpha,\sigma}r_{\prid,\sigma}\ \gcan
                              (P_{\alpha}^{\sigma}, Q_{\prid}^{\sigma}),
\end{equation*}
where $\gcan$ is the canonical Green's function on $\cX_{\sigma}$.
If $D$ and $E$ are divisors on $\cX$ we shall denote by $\ar{\ov D}{\ov E}$ the Arakelov intersection of 
the corresponding line bundles. To simplify expressions, we introduce the notation 
$\arsq{\ov{D}} := \ar{\ov D}{\ov D}$, for the Arakelov self-intersection of $D$. 

As in the introduction, we shall denote by $\omega_{\cX}$ the relative dualising sheaf of $\cX$ over $\spec R$,
(see Liu \cite[Chapter 6, Definition 4.7]{\liu}, where it is called the canonical sheaf). A canonical divisor
$\fK_{\cX}$ of $\cX$ is a divisor whose associated line bundle is the relative dualising sheaf. The 
adjunction formula for Arakelov theory comes up repeatedly in the following calculations, so let us state it here. 
If $r \in \spec R$ and $V$ is a prime vertical divisor contained in $\cX_r$ then by \cite[Chapter 9, Theorem 1.37]{\liu}
\begin{equation*}
  \ar{\ov{\fK}_{\cX}}{\ov V} = \big(2g(V) - 2\big) \log \#k(r) - \arsq{\ov{V}},
\end{equation*}
where $k(r)$ is the residue field at $r$ and $g(V)$ is the arithmetic genus of $V$. Furthermore if $H$ is a prime 
horizontal divisor, then by 
\cite[Chapter IV, Corollary 5.6]{\lang}
\begin{equation*}
  \ar{\ov{\fK}_{\cX}}{\ov H} = - \arsq{\ov{H}}.
\end{equation*}

\subsection{Stable arithmetic self-intersection number of $X_0(p^2)$}
Recall that the cusps $0, \infty$ of $X_0(p^2)$ are both $\qq$ points of the curve. Let $H_0$ and $H_{\infty}$ be the 
corresponding sections (horizontal divisors) in $\cX_0(p^2)/\cO_K$. Here we assume $p \neq 5, 7, 13$.

From \cite[Section 6]{ddc1} we know that the horizontal divisor $H_0$ intersects exactly one of the 
curves $\wt{C}_{2,0}$ or $\wt{C}_{0,2}$ of the special fiber --- at an $\ff_{p^2}$ rational point --- transversally 
(cf.  Liu~\cite[Chapter~9, Proposition~1.30 and Corollary~1.32]{MR1917232}). We call that component 
$\wt{C}_0$. It follows from the cusp and component labelling of Katz and Mazur~\cite[p.~296]{MR772569} that 
$H_{\infty}$ meets the other component transversally and we call it $\wt{C}_{\infty}$. 

Let us define two vertical divisors on $\cX_0(p^2)/\cO_K$
\begin{equation*}
  V_{m,p} = \sum_{i=1}^{\varphi(p+1)/2} V_{m,p}(\prid_i) \quad \text{for } m = 0, \infty
\end{equation*} 
supported on the special fibers, where $V_{m,p}(\prid_i)$ are as below. When $p=12k+1$ let 
$x = -(k-2)/k = -(p-25)/(p-1)$, define 
\begin{small}
\begin{align*}
  V_{0,p}(\prid_i) = 
            & \left(12 -12\genus\right) \wt{C}_0 + 7 \wt{C}_{1,1}^1 + 7\wt{C}_{1,1}^2 
              + \sum_{i=1}^k \frac{p-1}{2} x \wt{L}_i\\
            & + \sum_{i=1}^k \sum_{l=1}^{6k-1} 
              \left[lx + \frac{p-1-2l}{p-1}\left(12 -12\genus\right)\right] A_{l,i}
              + \sum_{i=1}^k \sum_{l=1}^{6k-1} (lx) B_{l,i},
  \\[10pt]       
  V_{\infty,p}(\prid_i) = 
                 & \left(12 -12\genus\right) \wt{C}_{\infty} + 7 \wt{C}_{1,1}^1 + 7\wt{C}_{1,1}^2 
                   + \sum_{i=1}^k \frac{p-1}{2} x \wt{L}_i\\ 
                 & + \sum_{i=1}^k \sum_{l=1}^{6k-1} 
                   \left[lx + \frac{p-1-2l}{p-1}\left(12 -12\genus\right)\right] B_{l,i} 
                  + \sum_{i=1}^k \sum_{l=1}^{6k-1} (lx) A_{l,i}.                
\end{align*}
\end{small}
If $p=12k+5$ let $x = -3(k-1)/(3k+1) = -(p-17)/(p-1)$, define
\begin{small}
\begin{align*}
  V_{0,p}(\prid_i) = 
            & \left(12 -12\genus\right) \wt{C}_0 + 3 \wt{C}_{1,1}^1 + 3 \wt{C}_{1,1}^2
              + \sum_{i=1}^k \frac{p-1}{2} x \wt{L}_i + \frac{p-1}{2} x \wt{F} \\
            & - (4k-5) S_1 - (4k-5) T_1 - (2k-4) S_2 - (2k-4) T_2 
              + \sum_{i=1}^k \sum_{l=1}^{6k+1} (lx) B_{l,i} + \sum_{j=1}^{18k+5} \frac{jx}{3} H_j\\
            & + \sum_{i=1}^k \sum_{l=1}^{6k+1} 
              \left[lx + \frac{p-1-2l}{p-1}\left(12 -12\genus\right)\right] A_{l,i}
              + \sum_{j=1}^{18k+5} 
              \left[\frac{jx}{3} + \frac{3(p-1)-2j}{3(p-1)}\left(12 -12\genus\right)\right] G_j,
  \\[10pt]
  V_{\infty,p}(\prid_i) = 
                 & \left(12 -12\genus\right) \wt{C}_{\infty} + 3 \wt{C}_{1,1}^1 + 3 \wt{C}_{1,1}^2
                   + \sum_{i=1}^k \frac{p-1}{2} x \wt{L}_i + \frac{p-1}{2} x \wt{F} \\
                 & - (4k-5) S_1 - (4k-5) T_1 - (2k-4) S_2 - (2k-4) T_2 
                   + \sum_{i=1}^k \sum_{l=1}^{6k+1} (lx) A_{l,i} + \sum_{j=1}^{18k+5} \frac{jx}{3} G_j\\
                 & + \sum_{i=1}^k \sum_{l=1}^{6k+1} 
                   \left[lx + \frac{p-1-2l}{p-1}\left(12 -12\genus\right)\right] B_{l,i}
                   + \sum_{j=1}^{18k+5} 
                   \left[\frac{jx}{3} + \frac{3(p-1)-2j}{3(p-1)}\left(12 -12\genus\right)\right] H_j.                   
\end{align*}
\end{small}
If $p=12k+7$ let $x = -2(k-1)/(2k+1) = -(p-19)/(p-1)$, then define
\begin{small}
\begin{align*} 
  V_{0,p}(\prid_i) = 
            & \left(12 -12\genus\right) \wt{C}_0 + 4 \wt{C}_{1,1}^1 + 4 \wt{C}_{1,1}^2
              + \sum_{i=1}^k \frac{p-1}{2} x \wt{L}_i + \frac{p-1}{2} x \wt{E}\\
            & - (3k-5) U - (3k-5) V + \sum_{i=1}^k \sum_{l=1}^{6k+2} (lx) B_{l,i}
              + \sum_{j=1}^{12k+5} \frac{jx}{3} N_j\\
            & + \sum_{i=1}^k \sum_{l=1}^{6k+2} 
              \left[lx + \frac{p-1-2l}{p-1}\left(12 -12\genus\right)\right] A_{l,i}
              + \sum_{j=1}^{12k+5} 
              \left[\frac{jx}{2} + \frac{p-1-j}{p-1}\left(12 -12\genus\right)\right] M_j, 
  \\[10pt]
  V_{\infty,p}(\prid_i) = 
                 & \left(12 -12\genus\right) \wt{C}_{\infty} + 4 \wt{C}_{1,1}^1 + 4 \wt{C}_{1,1}^2
                   + \sum_{i=1}^k \frac{p-1}{2} x \wt{L}_i + \frac{p-1}{2} x \wt{E}\\
                 & - (3k-5) U - (3k-5) V + \sum_{i=1}^k \sum_{l=1}^{6k+2} (lx) A_{l,i}
                   + \sum_{j=1}^{12k+5} \frac{jx}{2} M_j\\
                 & + \sum_{i=1}^k \sum_{l=1}^{6k+2} 
                   \left[lx + \frac{p-1-2l}{p-1}\left(12 -12\genus\right)\right] B_{l,i}
                   + \sum_{j=1}^{12k+5} 
                   \left[\frac{jx}{2} + \frac{p-1-j}{p-1}\left(12 -12\genus\right)\right] N_j.
\end{align*}
\end{small}
If $p=12k+11$ let $x = -6k/(6k+5) = -(p-11)/(p-1)$, then define
\begin{small}
\begin{align*} 
  V_{0,p}(\prid_i) = 
            & \left(12 -12\genus\right) \wt{C}_0 + \sum_{i=1}^k \frac{p-1}{2} x \wt{L}_i 
              + \frac{p-1}{2} x \wt{E} + \frac{p-1}{2} x \wt{F} \\ 
            & - 3k U - 3k V - 4k S_1 -4k T_1 - 2k S_2 - 2k T_2
              + \sum_{i=1}^k \sum_{l=1}^{6k+4} (lx) B_{l,i} + \sum_{j=1}^{12k+9} \frac{jx}{2} N_j 
              + \sum_{j=1}^{18k+14} \frac{jx}{3} H_j\\
            & + \sum_{i=1}^k \sum_{l=1}^{6k+4} 
              \left[lx + \frac{p-1-2l}{p-1}\left(12 -12\genus\right)\right] A_{l,i}
              + \sum_{j=1}^{12k+9} 
              \left[\frac{jx}{2} + \frac{p-1-j}{p-1}\left(12 -12\genus\right)\right] M_j \\
            & + \sum_{j=1}^{18k+14} 
              \left[\frac{jx}{3} + \frac{3(p-1)-2j}{3(p-1)}\left(12 -12\genus\right)\right] G_j,
  \\[10pt]
  V_{\infty,p}(\prid_i) = 
                 & \left(12 -12\genus\right) \wt{C}_{\infty} + \sum_{i=1}^k \frac{p-1}{2} x \wt{L}_i 
                   + \frac{p-1}{2} x \wt{E} + \frac{p-1}{2} x \wt{F} \\ 
                 & - 3k U - 3k V - 4k S_1 -4k T_1 - 2k S_2 - 2k T_2 + \sum_{i=1}^k \sum_{l=1}^{6k+4} (lx) A_{l,i}
                   + \sum_{j=1}^{12k+9} \frac{jx}{2} M_j + \sum_{j=1}^{18k+14} \frac{jx}{3} G_j\\
                 & + \sum_{i=1}^k \sum_{l=1}^{6k+4} 
                   \left[lx + \frac{p-1-2l}{p-1}\left(12 -12\genus\right)\right] B_{l,i}
                   + \sum_{j=1}^{12k+9} 
                   \left[\frac{jx}{2} + \frac{p-1-j}{p-1}\left(12 -12\genus\right)\right] N_j\\
                 & + \sum_{j=1}^{18k+14} 
                   \left[\frac{jx}{3} + \frac{3(p-1)-2j}{3(p-1)}\left(12 -12\genus\right)\right] H_j.                   
\end{align*}
\end{small}

Let $\can$ be a canonical divisor of $\cX_0(p^2)/\cO_K$.

\begin{proposition} \label{prop:divisors}
  For   $m \in \{0,\infty\}$, the divisors 
  $D_{m,p} = \can - (2\genus-2)H_m + V_{m,p}$ are orthogonal 
  to all vertical divisors with respect to the Arakelov intersection pairing.
\end{proposition}

\begin{proof}
  We note that $H_0$ intersects $\wt{C}_0$ transversally and does not meet any other component
  of the special fiber. Similarly $H_{\infty}$ intersects $\wt{C}_{\infty}$ transversally and no other 
  component. We shall only discuss the case when $p = 12k + 1$ and $k > 1$, since all other cases are 
  analogous.
  
  Let $\mathfrak{q} \in \spec \cO_K$ be a non-zero prime such that $\mathfrak{q} \neq \prid_i$ for 
  $i = 1, \ldots, \varphi(p+1)/2$. 
  Let $V = \cX_0(p^2)_{\mathfrak{q}}$, then $\ar{\ovcan}{\ov V} = (2\genus - 2)\log(\# (\cO_K/\mathfrak{q}) )$, 
  by the adjunction formula. The horizontal divisor $H_m$ intersects $V$ transversally at an  
  $\cO_K/\mathfrak{q}$ rational point, hence $\ar{\ov H_m}{\ov V} = \log(\# (\cO_K/\mathfrak{q}) )$. 
  Finally as $\ar{\ov V}{\ov V_{m,p}} = 0$, the result follows. 
  It remains to check for prime vertical divisors supported on special fibers.
  
  If $S$ is a special fiber then for any prime vertical divisor $V$, $\ar{\ov V}{\ov S} = 0$.
  Moreover noting that $\cO_K/ (\prid_i) = \ff_{p^2}$, we can calculate the self-intersections of the components 
  of the special fiber $\cX_0(p^2)_{\prid_i}$:
  \begin{equation*}
    \arsq{\ov{\wt{C}}_0} = \arsq{\ov{\wt{C}}_{\infty}} = \arsq{\ov{\wt{C}}{}_{1,1}^1} = 
                           \arsq{\ov{\wt{C}}{}_{1,1}^2} = -k\log(p^2),
  \end{equation*}
  \begin{equation*}
    \arsq{\ov{\wt{L}}_i} = -4\log(p^2), \quad \arsq{\ov{A}_{l,i}} = \arsq{\ov{B}_{l,i}} = -2\log(p^2).
  \end{equation*}
  
  If $i \neq j$ and $V$ is a prime vertical divisor supported on $\cX_0(p^2)_{\prid_i}$, then clearly 
  $\ar{\ov V}{ \ov{V}_{m,p}(\prid_j)} = 0$. Thus using the adjunction formula we have
  \begin{equation*}
    \ar{\ov{D}_{m,p}}{\ov V} = (2g(V) - 2)\log(\#(\cO_K/\prid_i)) - \arsq{ \ov V} 
                               -(2\genus -2)\ar{\ov{H}_m}{\ov V}  + \ar{\ov{V}_{m,p}(\prid_i)}{\ov V}.
  \end{equation*}
  Fixing a prime ideal $\prid_i$ we verify that this quantity is indeed zero for all such $V$.
  \begin{align*}
    \ar{\ov{D}_{0,p}}{\ov{\wt{C}}_0} 
            = & (2g(\wt{C}_0) - 2)\log(p^2) - \arsq{\ov{\wt{C}}_0} -(2\genus - 2)\log(p^2)   
                + (12 -12\genus) \arsq{\ov{\wt{C}}_0}  \\
              & + \sum_{i=1}^k \left[-\frac{k-2}{k} + \frac{p-3}{p-1}\left(12 -12\genus\right)\right] 
                \ar{\ov{A}_{1,i}}{\ov{\wt{C}}_0}\\
            = & \Big((k-2) - (2\genus - 2) - (k-2) + (p-1)(\genus-1) - (p-3)(\genus-1) \Big)\log(p^2)
                = 0,\\[5pt]            
    \ar{\ov{D}_{0,p}}{\ov{\wt{C}}_{\infty}} 
              = & (2g(\wt{C}_0) - 2)\log(p^2) - \arsq{\ov{\wt{C}}_0} + 
                  \sum_{i=1}^k -\frac{k-2}{k} \ar{\ov{B}_{1,i}}{\ov{\wt{C}}_0} \\
              = & \Big((k-2) - (k-2)\Big)\log(p^2) = 0,\\[5pt]
    \ar{\ov{D}_{0,p}}{\ov{\wt{C}}{}^j_{1,1}} 
                   = & (2g(\wt{C}_{1,1}^j) - 2)\log(p^2) - \arsq{\ov{\wt{C}}{}_{1,1}^j} + 7\arsq{\ov{\wt{C}}{}_{1,1}^j}
                       + \sum_{i=1}^k -6(k-2) \ar{\ov{\wt{L}}_i}{\ov{\wt{C}}{}_{1,1}^j} \\
                   = & \Big((6k^2 - 6k+ 2 -2 +k) - 7k -6k(k-2)\Big)\log(p^2) = 0, \\[5pt]
    \ar{\ov{D}_{0,p}}{\ov{\wt{L}}_i} 
               = & (2g(\wt{L}_i)-2)\log(p^2) - \arsq{\ov{\wt{L}}_i} -6(k-2)\arsq{\ov{\wt{L}}_i} \\
                 & \quad + 7\ar{\ov{\wt{C}}{}_{1,1}^1}{\ov{\wt{L}}_i} + 7\ar{\ov{\wt{C}}{}_{1,1}^2}{\ov{\wt{L}}_i} \\
                 & \quad + \left(-\frac{(6k-1)(k-2)}{k} + \frac{2}{p-1}(12 -12\genus)\right)
                   \ar{\ov{A}_{6k-1,i}}{\ov{\wt{L}}_i} \\
                 & \quad -\frac{(6k-1)(k-2)}{k} \ar{\ov{B}_{6k-1,i}}{\ov{\wt{L}}_i} \\
               = & \left((12k - 2 + 4) + 24(k - 2) + 14 -2\frac{(6k-1)(k-2)}{k} 
                   - 2\frac{12k^2-3k-2}{k} \right)\log(p^2) = 0.     
  \end{align*}
  The rest are just similar calculations observing that $A_{l,i}$ and $B_{l,i}$ are projective lines and keeping 
  track of which components they intersect.
\end{proof}

\begin{proposition}
  With the notation from the previous proposition, the following equality holds:
  \begin{equation*}
  \begin{aligned}
    \arsq{\ov{\omega}_{\cX_0(p^2)/\cO_K}}
    = & -4\genus(\genus-1) \ar{\ov{H}_0}{\ov{H}_{\infty}} \\
      & + \frac{1}{\genus-1} 
        \left( \genus \ar{\ov{V}_{0,p}}{\ov{V}_{\infty,p}} - 
        \frac{\arsq{\ov{V}_{0,p}} + \arsq{\ov{V}_{\infty, p}}}{2} \right) 
        + [K:\qq] e_p;
  \end{aligned}
  \end{equation*}
  where 
  \begin{equation*}
    e_p=
    \begin{cases}
    0 & \text{if $p \equiv 11 \pmod {12}$,} \\
    O(\log(p^2))&  \text{if $p \not \equiv 11 \pmod {12}$. } \\
    \end{cases}
  \end{equation*}
\end{proposition}

\begin{proof}
  The proof of this proposition is completely analogous to the proof of Lemma 6.2 of \cite{ddc1}
  so we skip some of the details here.
  
  Let $\mathscr{D}_{m,p}$ be the line bundle associated to $D_{m,p}$ for $m = 0, \infty$. Let 
  $h_{m,p}$ be the N\'{e}ron-Tate height of $\mathscr{D}_{m,p}$ considered as a point of the Jacobian
  of $X_0(p^2)$. Since $D_{m,p}$ has degree 0 and is orthogonal to vertical divisors, a theorem of 
  Faltings-Hriljac \cite[Theorem 4]{\fal} yields
  \begin{equation*}
    \arsq{\ov{D}_{m,p}} = -2 [K:\qq] h_{m,p}.
  \end{equation*}
  The generic fiber of the line bundle corresponding to the divisor $E = D_{0,p} - D_{\infty,p}$ is 
  supported at cusps. Hence a theorem of Manin-Drinfeld \cite{MR0314846, MR0318157} says that it is 
  a torsion point of the Jacobian. The divisor $E$ again has degree zero and is orthogonal to vertical 
  divisors thus the result of Faltings-Hriljac and the vanishing of N\'{e}ron-Tate height at torsion 
  points implies $\arsq{\ov{E}} = 0$. Using this, a formal argument yields
  \begin{align*}
    \arsq{\ov{\omega}_{\cX_0(p^2)/\cO_K}} = & -4\genus(\genus-1) \ar{\ov{H}_0}{\ov{H}_{\infty}}
                                           +\frac{\genus}{\genus-1} \ar{\ov{V}_{0,p}}{\ov{V}_{\infty,p}} \\ 
                                         & -\frac{1}{2\genus-2} \Big(\arsq{\ov{V}_{0,p}} + \arsq{\ov{V}_{\infty, p}} \Big) 
                                           +\frac{[K:\qq]}{2}(h_{0,p} + h_{\infty,p}).
  \end{align*}
  
  It is shown in the proof of \cite[Lemma 6.2]{ddc1} that $\mathscr{D}_{m,p}$ is supported only
  at cusps and elliptic points of $X_0(p^2)$. N\'{e}ron-Tate height at cusps vanish whereas there is
  a bound on the N\'{e}ron-Tate of elliptic points due to Michel-Ullmo \cite[Section 6]{MR1614563}. 
  From that we get $e_p = \frac{1}{2}(h_{0,p} + h_{\infty,p}) = O(\log p)$.
\end{proof}

\begin{corollary}
  We have the following  asymptotic  formula
  \begin{equation*}
    \arsq{\ov{\omega}_{\cX_0(p^2)/\cO_K}} = - 4\genus(\genus-1) \ar{\ov{H}_0}{\ov{H}_{\infty}} 
                                            + \frac{p^3\varphi(p+1)}{8} \log p 
                                            + O\Big(p^2\varphi(p+1) \log p \Big).
  \end{equation*}
\end{corollary}  
  
\begin{proof}  
  From Proposition~\ref{prop:divisors}, it follows that 
  $\arsq{\ov{V}_{0,p}} = (2\genus - 2)\ar{\ov{H}_0}{\ov{V}_{0,p}} - \ar{\ov{\omega}_{\cX_0(p^2)/\cO_K}}{\ov{V}_{0,p}}$
  and 
  $\ar{\ov{V}_{0,p}}{\ov{V}_{\infty, p}} = 
    (2\genus - 2)\ar{\ov{H}_0}{\ov{V}_{\infty,p}} - \ar{\ov{\omega}_{\cX_0(p^2)/\cO_K}}{\ov{V}_{\infty,p}}$.
  Using this we have the following:
  \begin{equation*}
    \arsq{\ov{V}_{0,p}} = \arsq{\ov{V}_{\infty,p}}  = 
    \begin{cases}
    - \dfrac{(4p^4 - 43p^3 + 39p^2 + 423p + 729)}{24} \dfrac{\varphi(p+1)}{2} \log(p^2), 
                 & \qquad p \equiv 1 \pmod{12}, \\[10pt]
    - \dfrac{(4p^4 - 43p^3 + 71p^2 + 263p + 217)}{24} \dfrac{\varphi(p+1)}{2} \log(p^2), 
                 & \qquad p \equiv 5 \pmod{12}, \\[10pt]
    - \dfrac{(4p^4 - 43p^3 + 63p^2 + 303p + 321)}{24}\dfrac{\varphi(p+1)}{2} \log(p^2), 
                 & \qquad p \equiv 7 \pmod{12}, \\[10pt]
    - \dfrac{(4p^4 - 43p^3 + 95p^2 + 143p + 1)}{24}\dfrac{\varphi(p+1)}{2} \log(p^2),   
                 & \qquad p \equiv 11 \pmod{12},
    \end{cases}
  \end{equation*}  
  and
  \begin{equation*}
    \ar{\ov{V}_{0,p}}{\ov{V}_{\infty,p}} = 
    \begin{cases}
    \dfrac{(3p^3 - 99p^2 + 377p + 871)}{24} \dfrac{\varphi(p+1)}{2} \log(p^2), & \qquad p \equiv 1 \pmod{12}, \\[10pt]
    \dfrac{(3p^3 - 67p^2 + 217p + 359)}{24} \dfrac{\varphi(p+1)}{2} \log(p^2), & \qquad p \equiv 5 \pmod{12}, \\[10pt]
    \dfrac{(3p^3 - 75p^2 + 257p + 463)}{24} \dfrac{\varphi(p+1)}{2} \log(p^2), & \qquad p \equiv 7 \pmod{12}, \\[10pt]
    \dfrac{(3p^3 - 43p^2 + 97p + 143)}{24}  \dfrac{\varphi(p+1)}{2} \log(p^2), & \qquad p \equiv 11 \pmod{12}.
    \end{cases}
  \end{equation*}
  Plugging these values in the formula of the previous proposition we obtain the
  present corollary.
\end{proof}

Now we can present a precise asymptotic expression for the main quantity of this section.

\begin{theorem} \label{thm:self-inter}
  The stable arithmetic self-intersection of $X_0(p^2)/K$ has the following asymptotic expression 
  \begin{equation*}
    \selfinter = 2 \genus\log (p^2) + \frac{p}{8} \log(p^2) + o(p\log (p^2)).                     
  \end{equation*}  
\end{theorem}

\begin{proof}
  Since $[K:\qq] = \varphi(p+1)(p^2-1)/2$, we have
  \begin{equation*}
    \selfinter = \frac{1}{\varphi(p+1)(p^2-1)/2} \arsq{\ov{\omega}_{\cX_0(p^2)/\cO_K}}.
  \end{equation*}  
  Let $\gcan$ be the canonical Green's function as in Definition \ref{def:green}. The Arakelov intersection of  
  two inequivalent prime horizontal divisors is given completely by the canonical Green's function, thus 
  (cf. \S~\ref{Arakelovintersection} and  \cite[p. 59] {Grados:Thesis})
  \begin{equation*}
    \ar{\ov{H}_0}{\ov{H}_{\infty}} = -\sum_{\sigma: K \to \cc} \gcan(0^{\sigma},\infty^{\sigma}) .
  \end{equation*}
  Note that the cusps $0$ and $\infty$ of $X_0(p^2)$ are both rational points of the modular curve and hence we have
  \begin{equation*}
   \gcan(0^{\sigma},\infty^{\sigma}) =  \gcan(0, \infty), \quad \text{ for all embeddings } \sigma: K \to \cc.
  \end{equation*}     
  Hence from the previous corollary we obtain 
  \begin{equation*}
    \selfinter = 4\genus(\genus-1)\mathfrak{g}_{can}(0,\infty) + o(p^2\log p).               
  \end{equation*}   
  From \cite[Proposition~4.9]{ddc1}, for $p>7$
  \begin{equation*}
    \gcan(0,\infty) = \frac{\log p}{\genus} + o\left(\frac{\log p}{\genus}\right).
  \end{equation*}  
  This completes the proof.  
\end{proof}

\subsection{Stable Faltings height}

Let $\mathrm{Jac}(X_0(N))/\qq$ be the Jacobian variety of the modular curve $X_0(N)/\qq$. We denote by 
$\hfal{\mathrm{Jac}(X_0(N))}$ the stable Faltings height of $\mathrm{Jac}(X_0(N))/\qq$. 

\begin{corollary} \label{cor:falt}
  The stable Faltings height of the modular curve $X_0(p^2)$ satisfies the following asymptotic estimate:
  \begin{equation*}
    \hfal{\jac} = \frac{1}{6} \genus\log(p^2)
                      + o(\genus\log(p^2)).
  \end{equation*}
\end{corollary}

\begin{proof}  
  The arithmetic Noether formula (see Moret-Bailly \cite[Theorem 2.5]{\mor})
  implies 
  \begin{equation} \label{eq:noether}
    12\hfal{\jac} = \selfinter + \sum_{i = 1}^{\varphi(p+1)/2} \frac{2 s(\prid_i) \log (p)}{\varphi(p+1)(p^2 - 1)}  
                        + \dfal(X_0(p^2)) - 4\genus \log (2\pi).
  \end{equation}

  Here $s(\prid_i)$ denotes the number of singular points in the special fiber of 
  $\cX_0(p^2)_{\prid_i}$ and $\dfal$  is the Faltings delta invariant, as 
  defined in \cite[Theorem 1]{\fal}. From Section \ref{sec:semistable},  
  we calculate $s(\prid_i)$:
  \begin{align*}
    s(\prid_i) = 
    \begin{cases}
      \frac{p^2-1}{12}, & \qquad p \equiv 1 \pmod{12},       \\
      \frac{(p+1)(p+31)}{12}, & \qquad p \equiv 5 \pmod{12}, \\ 
      \frac{(p+1)(p+17)}{12}, & \qquad p \equiv 7 \pmod{12}, \\   
      \frac{(p+1)(p+49)}{12}, & \qquad p \equiv 11 \pmod{12}.
    \end{cases}  
  \end{align*}
  Thus, we deduce that:
  \begin{equation*}
    \sum_{i = 1}^{\varphi(p+1)/2} \frac{2 s(\prid_i) \log (p)}{\varphi(p+1)(p^2 - 1)} = o(\log(p^2)).
  \end{equation*}
  
  By \cite[Theorem 5.6]{\jorg}, we have $\dfal(X_0(p^2)) = O(\genus) = O(p^2)$ 
  if $\genus > 1$ (valid for $p>7$). Hence the result follows from 
  \eqref{eq:noether} and Theorem \ref{thm:self-inter}.
\end{proof}

\section{An effective Bogomolov conjecture} \label{sec:bogomolov}
\newcommand{\adselfinter}{\ov\omega_{a, \cX_0(p^2)}^2}
\newcommand{\graph}{G_{\prid_i}}

We assume that the reader is familiar with the theory of admissible pairing of Zhang \cite{\zhang}.
The goal of this section is to prove Theorem~\ref{thm:effectivebogomolov}.  The proof requires us to  
calculate the admissible self-intersection of the relative dualising sheaf $\omega_{\cX_0(p^2)/\cO_K}$. 
The admissible self-intersection is related to the stable arithmetic self-intersection number through 
the geometry of the special fibers of $\cX_0(p^2)/\cO_K$. This relation can be best captured 
using metrized dual graphs corresponding to the special fiber. So let us start by drawing the dual graphs.

\subsection{Dual graphs}

Let $\graph$ be the dual graph of the special fiber $\cX_0(p^2)_{\prid_i}$. This is a metrized
graph in the sense of Zhang \cite{\zhang}. Let $V(\graph)$ be the set of vertices, and 
$E(\graph)$ the set of edges of $\graph$. We denote by $l(\graph)$ the sum of lengths of all the edges. For any 
$x \in V(\graph)$, $g(x)$ is the geometric genus of the corresponding component of $\cX_0(p^2)_{\prid_i}$.
The figures below show the dual graphs of the special fibers of $\cX_0(p^2)$. The vertices in the dual graphs
will be labelled by the corresponding components of the special fiber, although we use small letters here
and ignore the tildes. We ignore the vertices with genus 0 and valence 2.

\begin{tabular}{ l m{0.35\textwidth} }
  \begin{minipage}{0.55\textwidth}
    \includegraphics[scale=0.6]{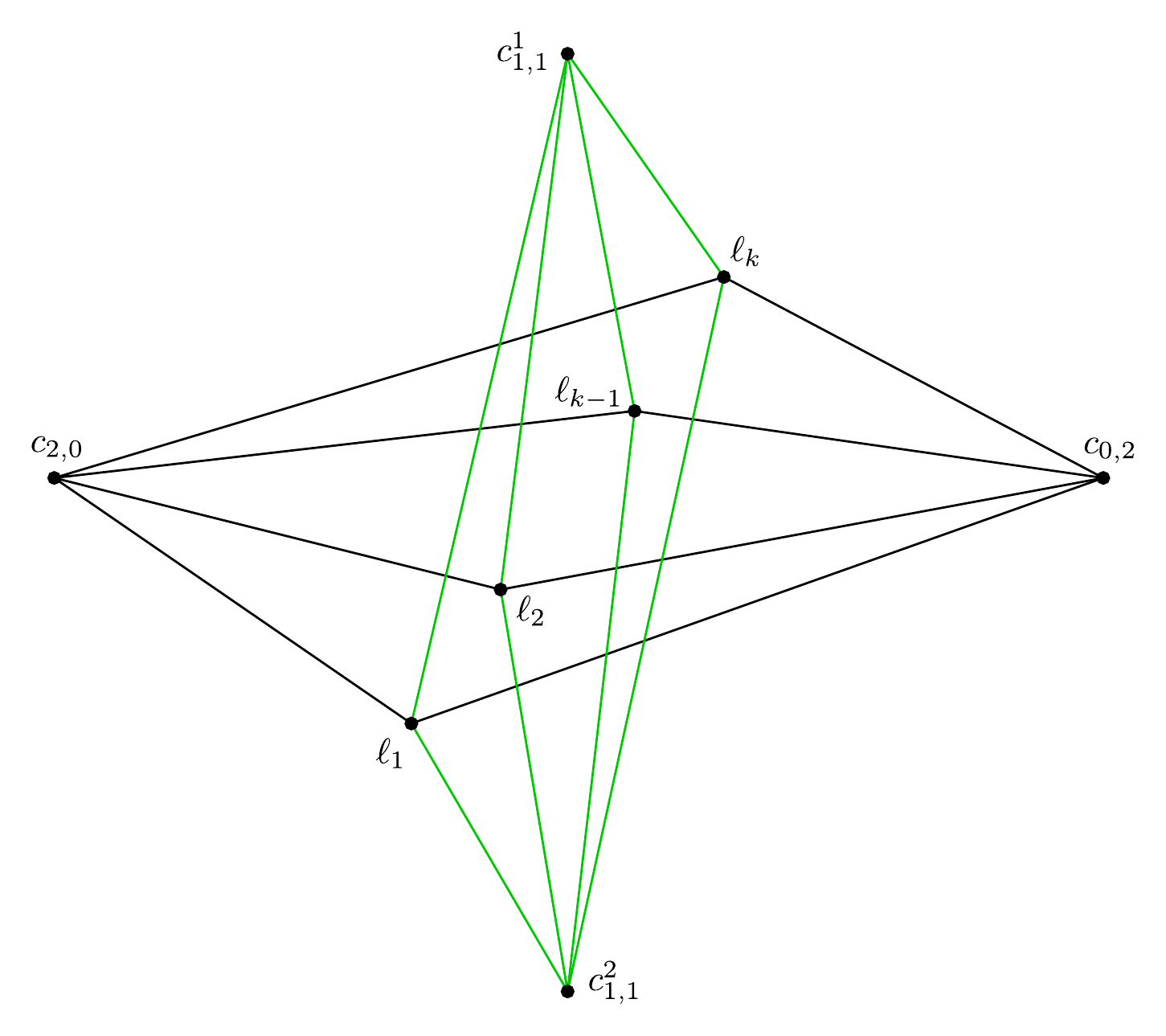} 
  \end{minipage}
  &
  \begin{minipage}{0.35\textwidth}
    $p = 12k+1$: \\
    Each black edge has length $6k$ and green edge has 
    length $1$, hence $l(\graph)=12k^2+2k$. Moreover \\
    $g(c_{2,0}) = g(c_{0,2}) = 0, \\
    g(c_{1,1}^i) = 3k^2 -3k + 1, \ g(\ell_i) = 6k$.
  \end{minipage}  
  \\ 
  
  \begin{minipage}{0.55\textwidth}
    \includegraphics[scale=0.6]{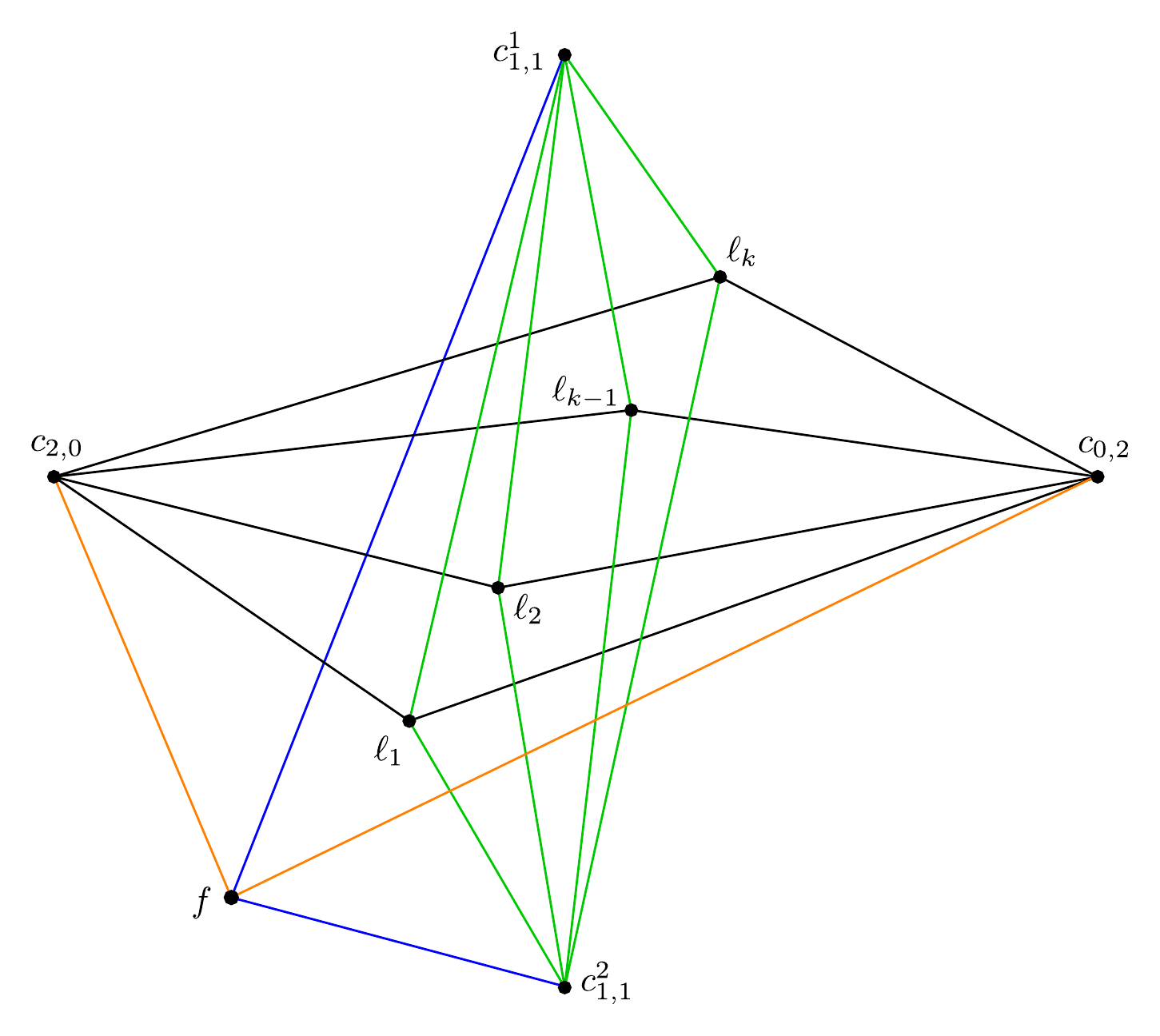} 
  \end{minipage}
  &
  \begin{minipage}{0.35\textwidth}
    $p=12k+5$: \\ 
    Each black edge has length $6k+2$, orange edge 
    length $18k+6$, blue edge length $3$ and green 
    edge length $1$, hence $l(\graph)=12k^2+42k+18$.
    Moreover \\
    $g(c_{2,0}) = g(c_{0,2}) = 0,\\ 
    g(c_{1,1}^i) = 3k^2 - k, \\
    g(\ell_i) = 6k+2,\ g(f)= 2k.$     
  \end{minipage}  
\end{tabular}

\begin{tabular}{ c m{0.35\textwidth} }
  \begin{minipage}{0.55\textwidth}
    \includegraphics[scale=0.6]{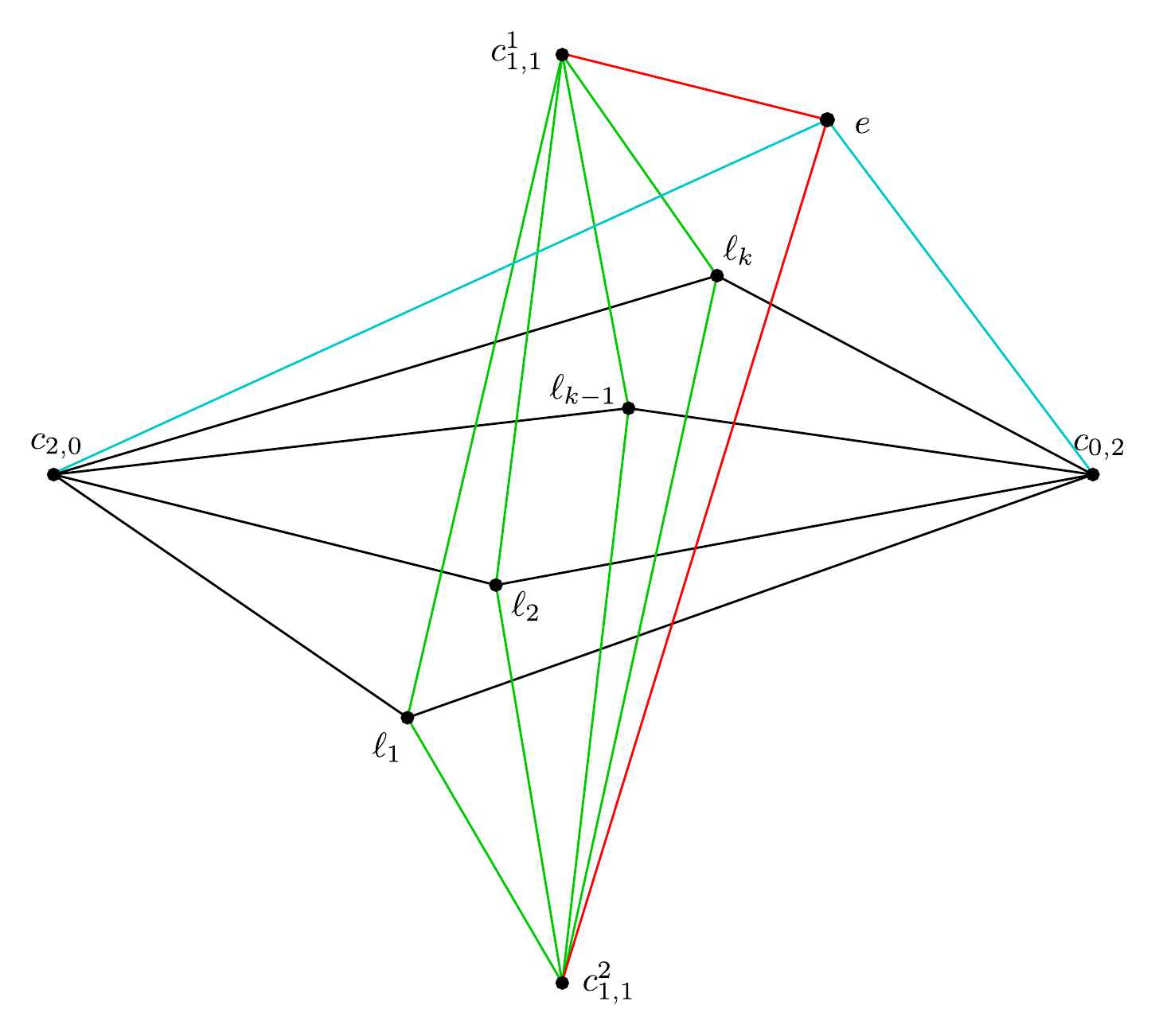}
  \end{minipage}
  &
  \begin{minipage}{0.35\textwidth}
    $p = 12k + 7$: \\
    Each black edge has length $6k+3$, cyan edge length 
    $12k+6$, red edge length 2 and green edge length 
    $1$, hence $l(\graph)=12k^2+32k+16$. Moreover \\
    $g(c_{2,0}) = g(c_{0,2}) = 0,\ 
    g(c_{1,1}^i) = 3k^2, \\ 
    g(\ell_i) = 6k+3,\ g(e)= 3k+1.$ 
  \end{minipage}  
  \\
  
  \begin{minipage}{0.55\textwidth}
    \includegraphics[scale=0.6]{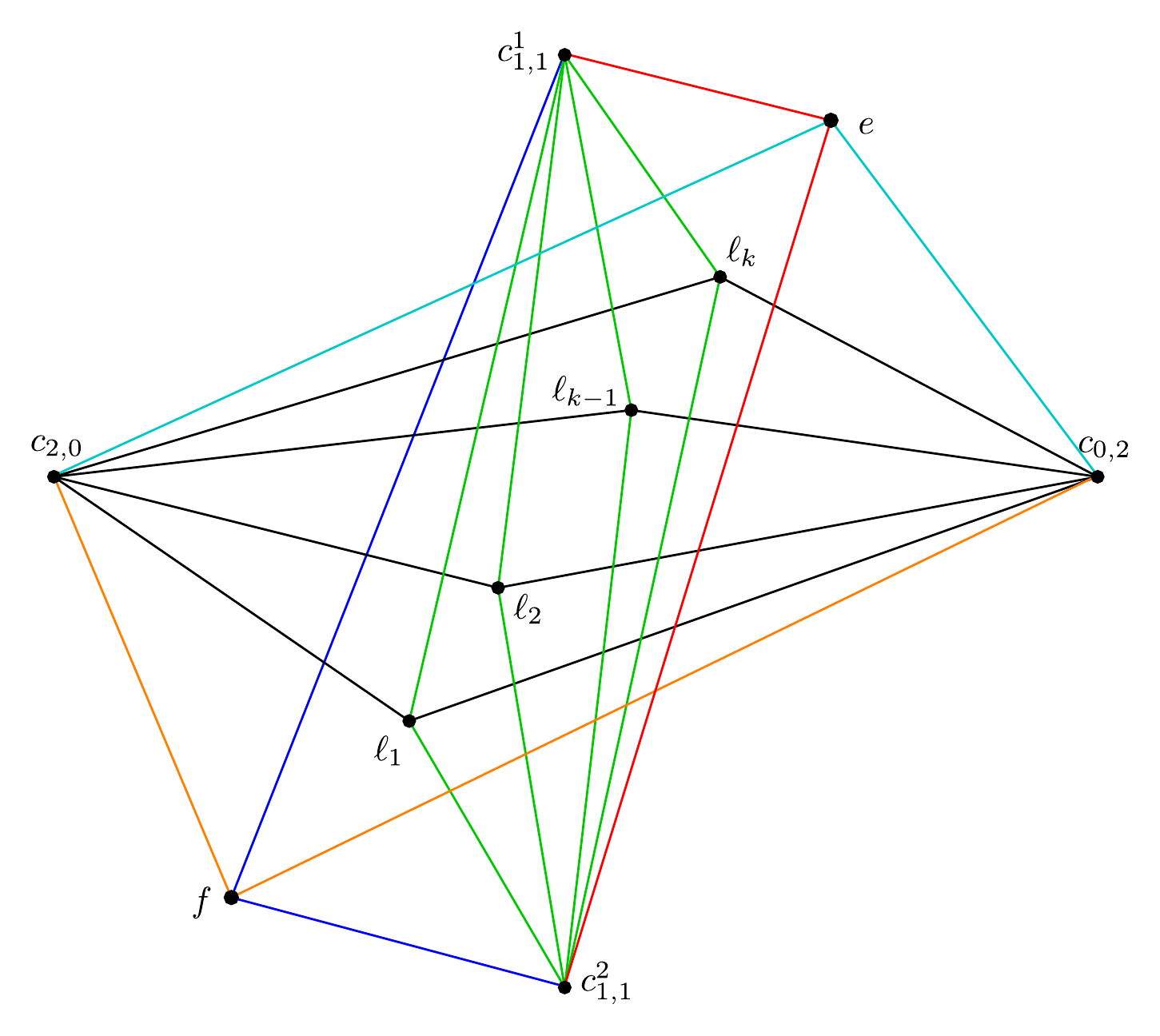}
  \end{minipage}
  &
  \begin{minipage}{0.35\textwidth}
    $p =12k + 11$: \\
    Each black edge has length $6k+5$, orange edge length 
    $18k+15$, cyan edge length $12k+10$, blue edge length 
    3, red edge length 2 and green edge length $1$,
    hence $l(\graph)=12k^2+72k+60$. \\ Moreover
    $g(c_{2,0}) = g(c_{0,2}) = 0, \\
    g(c_{1,1}^i) = 3k^2+2k,\ g(\ell_i) = 6k+5,\\
    g(e)= 3k+2,\ g(f) = 2k+1$.
  \end{minipage}
  \\
\end{tabular}

\subsection{Admissible pairing}

First we start with two propositions about the asymptotic behaviour of some quantities related to the  
dual graphs that we listed above. Let $\tau(\Ga)$ be the tau constant associated to a metrized graph $\Ga$  
as defined in \cite{Cinkir:Thesis}. 

\begin{proposition}
\label{lengthbound}
We have the following asymptotic estimate for $\tau(\graph)$ 
\[
\frac{8(\genus-1)}{(p^2-1)\genus^2}\tau(\graph) \rightarrow 0.
\]
\end{proposition}

\begin{proof}
Note that 
\[
  l(\graph)=12k^2+ak+b. 
\]
with $a,b \in \N$. 
Recall by \cite[Remark 2.1]{ddc1}, we  have 
\begin{equation} \label{eq:genus}
 \genus =1+\frac{(p+1)(p-6)-12c}{12}
\end{equation}
with $c \in \left\{1, \dfrac{1}{2}, \dfrac{1}{3}, \dfrac{1}{6}\right\}$.
From the above expression, it is easy to see that $\dfrac{8(\genus-1)}{(p^2-1)\genus^2}l(\graph) \rightarrow 0$ 
as $p \rightarrow \infty$. 

Recall that by \cite[Equation 14.3, p. 37]{MR2310616}, if $\Gamma$ is a graph with $n$ 
edges, then we have a bound on $\tau(\Gamma)$: 
\[
\frac{1}{16n } l(\Gamma) \leq  \tau(\Gamma) \leq \frac{1}{4}  l(\Gamma). 
\]
Hence, we obtain the asymptotic bound as in the proposition. 
\end{proof}

For any two vertices $x$ and $y$ of $\graph$, let $r(x,y)$ be the resistance between them. The resistance 
function is defined in Zhang \cite[Section 3]{\zhang}; see Proposition 3.3 of Zhang for properties of $r$.
Let us  denote
\[
\theta(x,y)=(v(x)-2+2 g(x)) (v(y)-2+2 g(y)) r(x,y). 
\] 
We also consider the following quantity associated to any metrized 
graph $\Ga$, as defined in \cite{Cinkir:Thesis}
\[
\tilde{\theta}(\Gamma)=\sum_{x,y \in V(\Gamma)} \theta(x, y).
\]

\begin{proposition}
\label{boundpdifcong}
We have the following asymptotic bound
\begin{equation*}
  \frac{1}{(p^2-1)\genus^2}\tilde{\theta}(\graph) = O\left(\frac{1}{p^2}\right).
\end{equation*}
\end{proposition}

\begin{proof} 
 For any two vertices $x,y$ of a metrized graph $\Ga$, $r(x,y)$ is bounded above by the length 
of the shortest path between the vertices  \cite[p. 15, Exercise 12]{MR2277605}. Let $p = 12k+1$, then
\begin{equation*}
\begin{aligned}
  \theta(l_i,l_j)  & \leq 2(12k + 2)^2, \\ 
  \theta(l_i, c_{2,0}) = \theta(l_i, c_{0,2}) & \leq (k - 2)(12k + 2)(6k), \\ 
  \theta(l_i, c_{1,1}^1) = \theta(l_i, c_{1,1}^2)  & \leq (12k + 2)(6k^2 - 5k), \\
  \theta(c_{2,0},c_{0,2}) & \leq (12k)(k-2)^2, \\
  \theta(c_{1,1}^i,c_{0,2}) = \theta(c_{1,1}^i,c_{2,0}) & \leq (6k-1)(k-2)(6k^2-5k), \\
  \theta(c_{1,1}^1,c_{1,1}^2) & \leq 2(6k^2-5k)^2.
\end{aligned}  
\end{equation*}
The result now follows by summing everything up and dividing by $\genus^2(p^2-1)$ and noting that
$\genus = O(p^2)$. In all other cases the proof is similar.
\end{proof}

Let $\ov\omega_{a, \cX_0(p^2)/\cO_K}$ be the admissible metrized relative dualising sheaf of the curve $\cX(p^2)/\cO_K$; 
see Zhang \cite{\zhang} for definition.

\begin{lemma}
\label{localadmissiblepairing1}
We have the following asymptotic expression for admissible self-intersection 
numbers for the modular curves $X_0(p^2)$:
\[
  \adselfinter = 2 \genus \log(p^2)+o(\genus \log(p)).
\] 
\end{lemma}

\begin{proof}
Recall the following Theorem that connects admissible self-intersection number with the arithmetic 
self-intersection number  \cite[Theorem 4.45]{Cinkir:Thesis}, \cite[Theorem 5.5]{\zhang}:
\[
  \adselfinter=\selfinter - \frac{1}{[K:\qq]} \sum_{i=1}^{\varphi(p+1)/2} 
  \left(\frac{4(\genus-1)}{\genus}\tau(\graph)+\frac{1}{2\genus} \tilde{\theta}(\graph)\right) \log p. 
\]
Hence the result now follows from Theorem~\ref{thm:self-inter}, and Propositions \ref{lengthbound} and 
\ref{boundpdifcong}.
\end{proof}

\begin{proof}[Proof of Theorem~\ref{thm:effectivebogomolov}]
Since $\hNT(x) \geq 0$ for all $x$, we note that by ~\cite[Theorem 5.6] {\zhang}
\[
  a'(D) \geq  
  \frac{\adselfinter}{4(\genus-1)}
  + \frac{2 \genus - 2}{g_p^2} \hNT\left(D - \frac{\fK_{X_0(p^2)}}{2\genus-2}\right)
  \geq \frac{1}{2} a'(D)
\]
for the divisor $D=\infty$ as in the introduction and $\fK_{X_0(p^2)}$ a canonical divisor of 
the Riemann surface $X_0(p^2)$. Let $h = \hNT\left(D - \dfrac{\fK_{X_0(p^2)}}{2\genus-2}\right)$, then by an argument 
similar to Section 6 of Michel-Ullmo \cite{MR1614563}, we can show that
\begin{equation} \label{eq:height}
  h = \frac{1}{\genus^2}O(\log p).
\end{equation}

By Zhang's theorem 
\[
 \hNT(\phi_D(x)) \leq a'(D)-\epsilon
\]
holds for only finitely many $x$. We deduce that
\[
  F_2=\left\{x \in X_0(p^2)(\ov\qq) \mid \hNT (\varphi_D(x)) 
  < \frac{\adselfinter}{4(\genus-1)} + h -\epsilon \right \}
\]
is finite. 

On the other hand, the set 
\[
  F_1=\left\{x \in X_0(p^2)(\ov\qq) \mid \hNT (\varphi_D(x)) 
  \leq  \frac{\adselfinter}{2(\genus-1)} + 2h + \epsilon \right\}
\]
is infinite. From Lemma~\ref{localadmissiblepairing1}, we conclude that for large enough $p$: 
\[
  2 - \epsilon < 
  \frac{\adselfinter}{\genus \log(p^2)}
  < 2 + \epsilon.
\]

The result now follows because of the bound on $h$ given by \eqref{eq:height}.
\end{proof}

\appendix
	\section{Small primes}
For the prime $5$, $X_0(25)$ has genus 0.
The next prime is $7$ in which case $g(X_0(49))$ is 1. From the calculations of section 
\ref{subsec:7mod12} it follows that the special fibers of the minimal regular model of $X_0(49)$ 
over $\cO_K$ is a single genus 1 curve hence in fact smooth, this is the component $\wt{E}$, all 
other components can be blown down.

For $p = 11$ the regular minimal model of $\cX_0(121)/\cO_K$ is as described in 
\ref{subsec:11mod12}.

Finally when $p =13$ we see that the special fibers of $\cX_0(169)/\cO_K$ is described by the following 
diagram. We have to contract $\wt{C}_{2,0}$ and $\wt{C}_{0,2}$ since these are genus $0$ 
components with self-intersection $-1$. 
\begin{center}
\includegraphics[scale=.5]{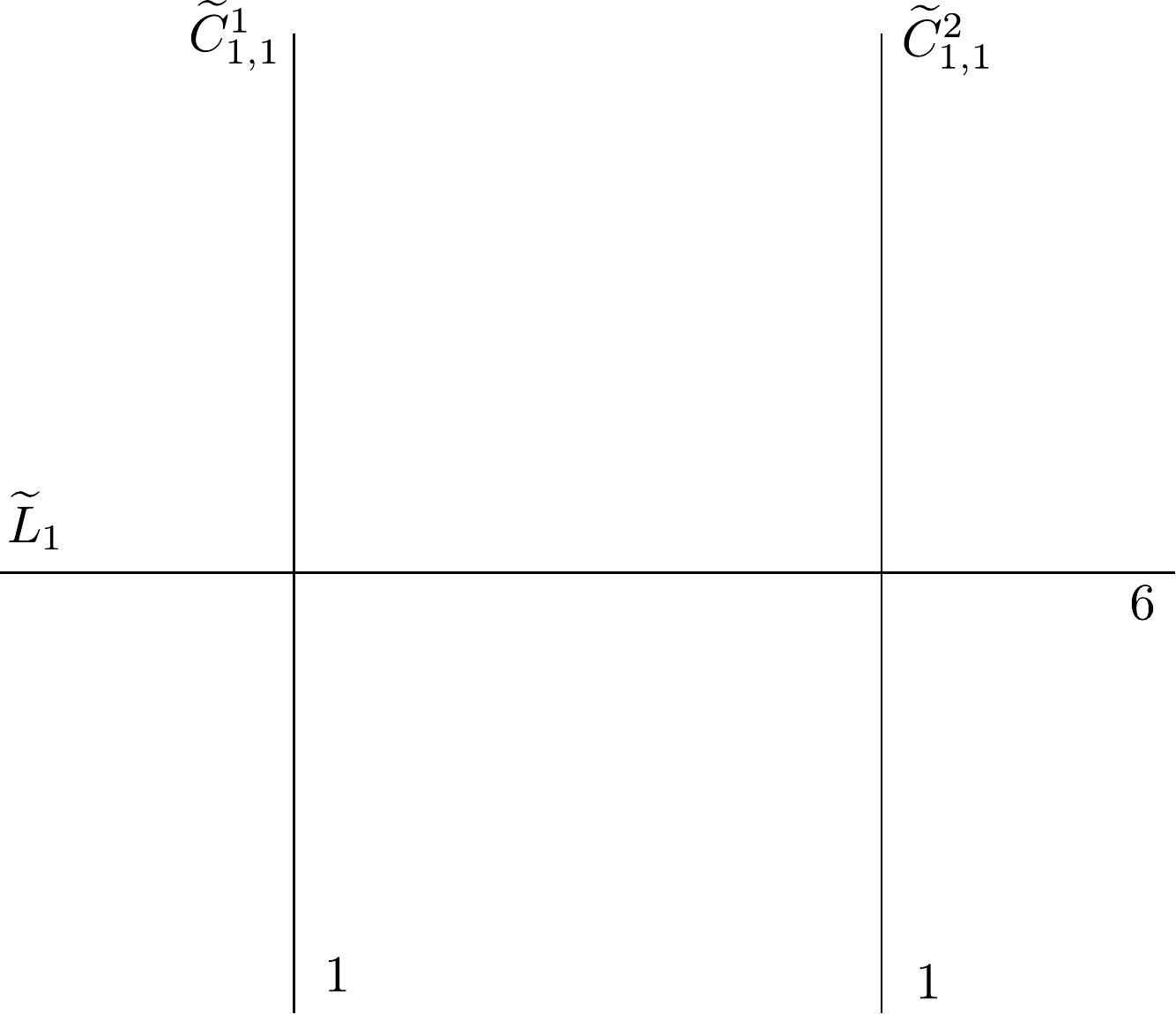}
\end{center}

\bibliographystyle{crelle}
\bibliography{Eisensteinquestion.bib}
\end{document}